\documentclass[12pt,a4paper]{article}
\usepackage{mathrsfs}
\usepackage{amsfonts}
\usepackage{cases}\usepackage{latexsym}
\usepackage{times}
\usepackage{dsfont}
\usepackage{color}
\usepackage{graphics}\usepackage{amsthm}
\newenvironment{keywords}{{\bf Keywords: }}{}

\newtheorem{remark}{Remark}[section]
\newtheorem{lemma}{Lemma}[section]

\newtheorem{theorem}{Theorem}[section]
\newtheorem{definition}{Defintion}[section]

\begin{document}
\author{Ping Lin\thanks{School of Mathematics and Statistics of Northeast Normal University, Changchun, 130024, P. R. China.(linp258@nenu.edu.cn).}
\and Hanbing Liu\thanks{ School of Mathematics and Physics of China University of
Geoscience, Wuhan, 430074, P.R.China.({
hanbing272003@aliyun.com}).}
\and Gengsheng Wang \thanks{Center for Applied Mathematics, Tianjin University, Tianjin, 300072, P. R. China.(wanggs62@yeah.net).}}

\title
 {Output feedback  stabilization for heat equations with sampled-data controls}
 \date{}

\maketitle
\begin{abstract}
In this paper, we build up an  output feedback law
to stabilize a sampled-data controlled   heat equation (with a potential) in a bounded domain $\Omega$. The feedback law
abides   the following  rules: First, we divide equally the time interval $[0,+\infty)$ into infinitely many disjoint  time periods, and  divide  each  time period  into three disjoint subintervals.
 Second, for each  time period, we  observe a solution over an open subset of $\Omega$
    in the first subinterval;  take sample from outputs at one time point of the first subinterval;  add a time-invariant output feedback control over another  open subset of $\Omega$ in the second subinterval;
   let the equation evolve free in the last subinterval. Thus, the corresponding
 feedback control is of  sampled-data. Our feedback law has the following advantages:
 the sampling period (which is the length of the above time period) can be arbitrarily taken;
  the feedback law has an explicit expression in terms of the sampling period;
   the behaviors of the norm of the feedback law, when the sampling period goes to zero or infinity, are clear.
The construction of the feedback law is based on two  kinds of approximate null-controllability
for heat equations. One has  time-invariant controls, while another has  impulse controls. The studies of the aforementioned controllability with
time-invariant controls need  a new   observability inequality
for heat equations built up in the current work.

\end{abstract}

\section{Introduction}
\label{}

 We start with introducing the controlled equation. Let $\Omega\subset \mathds{R}^d, d\geq 1$, be a bounded domain with a $C^2$ boundary $\partial \Omega$. Let $V(\cdot)$ be a nonzero function in $L^\infty(\Omega)$ with its norm $\|\cdot\|_\infty$. Define an operator $A: D(A)(:=H_0^1(\Omega)\cap H^2(\Omega))\rightarrow L^2(\Omega)$ by $Ay:=\triangle y-Vy$ for each $y\in H_0^1(\Omega)\cap H^2(\Omega)$.
Let $\omega\subset\Omega$ be an  open and nonempty  subset
and let $\mathds{1}_\omega$ be the zero extension operator from $L^2(\omega)$ to $L^2(\Omega)$.   Consider the following  controlled equation:
\begin{equation}\label{e101}
 \left\{\begin{array}{ll}
 y'(t)- Ay(t)=\mathds{1}_\omega u(t), \;\; t>0,\\
  y(0)\in L^2(\Omega),
\end{array}\right.
\end{equation}
where $u(\cdot)\in L^\infty(0, +\infty; L^2(\omega))$ is a sampled-data control.
We assume that the operator $-A$ has  at least one negative eigenvalue. Then the equation (\ref{e101}) with the null control is unstable.
\subsection{Notations}
The following notations are used throughout this paper:  Write $\|\cdot\|$ and $\langle \cdot, \cdot\rangle$  for the norm and the inner product of $L^2(\Omega)$ respectively. The norm and the inner product of $L^2(\omega)$ are denoted by $\|\cdot\|_\omega$ and $\langle \cdot, \cdot\rangle_\omega$ respectively.  Write $\mathds{R}^+=(0, +\infty)$, $\mathds{N}=\{0,1,2,\dots\}$ and $\mathds{N}^+=\{1,2,3,\dots\}$.
The usual norm in $\mathds{R}^d$, with $d\in \mathds{N}^+$,  is denoted by $|\cdot|_d$.
Let
   \begin{equation}\label{wang1.2}
   \gamma_0=\|V(\cdot)\|_{L^\infty(\Omega)}.
   \end{equation}
Let $\{\lambda_j\}_{j=1}^\infty$, with
\begin{equation}\label{wanggs1.3}
\lambda_1\leq\lambda_2\leq\cdots \leq\lambda_m\leq 0<\lambda_{m+1}\leq\cdots
\end{equation}
 be the family of all  eigenvalues of $-A$ and let
 $\{\xi_j\}_{j=1}^\infty$ be the family of the corresponding
 eigenfunctions, which forms an  orthonormal basis of $L^2(\Omega)$.
        For each $M\in\mathds{N}^+$,  let $X_M=\mathrm{span}\{\xi_i\}_{i=1}^M$, and let $P_M$
   be the orthogonal projection from $L^2(\Omega)$ onto $X_M$. Given a set $E$ in a  topological space $F$, write $\chi_E(\cdot)$ for its characteristic function.  Given a matrix $B$, we use  $B'$  and $B^{-1}$ to denote its transposition and inverse respectively.
      Given two  functions as $s\rightarrow s_0$ for some $s_0\geq 0$ (or  $s_0=+\infty$), the notation $f(s)=O(g(s))$ means that
   $\lim_{s\rightarrow s_0}f(s)$ and $\lim_{s\rightarrow s_0}g(s)$ have the same order.
   Write  $C(\cdots)$ for a positive constant depending on what are enclosed  in the brackets.

\subsection{Aims, previous works and motivations}

 {\bf Aims.}   The first aim of this paper is to design an output feedback law to stabilize exponentially the equation (\ref{e101}) with the feedback control. The feedback law abides   the following  rules: Let $T>0$. Write $[0, +\infty)=\bigcup_{i=0}^{+\infty} [2iT, (2i+2)T)$.
 Let $\omega_1\subset\Omega$ be another open and nonempty subset.
 In each subinterval $[2iT, (2i+2)T)$, we do what follows: First, in $[2iT, (2i+1)T)$, observe a solution $y(\cdot)$
 over   $\omega_1$; Second, take sample from the output $\mathds{1}_{\omega_1}^*y(\cdot)$ at the time instant $\frac{3}{4}T+2iT$; (Here, $\frac{3}{4}T$ can be any time before $T$.) Third, in $[(2i+1)T, (2i+\frac{3}{2})T)$,
  apply over $\omega$
 a time-invariant control, which is the feedback of the sampled output $\mathds{1}^*_{\omega_1}y(3/4T+2iT)$; Finally, in $[(2i+\frac{3}{2})T, (2i+2)T)$,
  let the equation evolve free. In such a way, the constructed  feedback control is
 of sampled-data and $2T$ can be viewed as  the sampling period.

 The second aim of this paper is to study   the behaviors of the norm of the feedback law, when the sampling period
 tends to zero or infinity.

  \noindent {\bf Some previous works.}
There have been many literatures on the stabilization of distributed parameter systems. Most of them are concerned about constructing continuous-time state feedback laws (see \cite{05, 06, 22, 23} and the references therein).
However, in  real applications,
one cannot obtain
  full information on  states in many cases.
  Thus,  the studies on output feedback stabilization should be very important.
  Unfortunately,   there has been no systematic theory on this subject. Most of publications on this issue focus on how to construct  output feedback  laws
  to stabilize some special equations (see \cite{24, 25, 26, 27, 01} and references therein). The current paper also only provides a way to design an output feedback law.

  Nowadays, most controls  are implemented in systems via digital technology. This motivates studies on
  continuous-time process with discrete-time (or piecewise constant) controls. We call such a process as a
   sampled-data controlled system and call such a control as a sampled-data control.
   (See, for instance, \cite{37}.)

   Stabilization problems for parameter distributed systems with sampled-data controls have  been studied in some literatures. In  \cite{09} and \cite{14}, some sufficient conditions on state and output feedback stabilization for some  abstract linear  systems with sampled-data controls were given respectively. In \cite{07} and \cite{13}, some output feedback laws for a class of
    1-D semilinear parabolic equations with sampled-data controls
   were  designed. In these two works,
   the space interval was divided into sufficiently small subintervals. The observation in \cite{07} is made on the endpoints of these subintervals, while the output in \cite{13} is the averaged state measurement on each subintervals.
        In \cite{28}, two types
        of feedback laws to stabilize 1-D  linear parabolic equations with boundary sampled-data controls were given
        by the emulation of the reduced model design and the emulation of the back-stepping design respectively. All of the above works  start with designing feedback laws  stabilizing  continuous-time systems with  continuous-time controls, and then
        design the feedback sampled-data controls by using  the aforementioned feedback laws   and the time-discretizing outputs.
         Such method is called the indirect method.
        Instead of using the indirect method, the works \cite{04} used approximate models to construct the sampled-data feedback stabilization laws.
        The method to design the feedback controls  in \cite{04}  is as follows: First,  discretize  continuous-time  state systems (including states, outputs and controls) in time; Second,  design  feedback laws for  discretizing state systems; Finally,  let the aforementioned discretizing feedback laws play the roles of
        sampled-data feedback laws for the original systems. All these  works
         need that sampling periods are sufficient small.
        This requirement is from  their methods  designing
        the feedback laws.

        We  mention  works   \cite{17, 10, 11}, where feedback laws are not static but time-periodic, and the sampling period is not necessary to be sufficient small.
        In \cite{17}, a type of time-periodic  output
        feedback        law
        was built up to  stabilize a type of neutral systems with sampled-data controls. The works \cite{10} and \cite{11} considered a generalized sampling operator, which is a kind of weighted average of the output over each sampling period. Some necessary and sufficient conditions for the stabilization of a class of linear distributed system were studied in these two works. It deserves to mention what follows:
        Since time-periodic feedback laws are much more powerful than time-invariant feedback laws (see \cite{23}), the authors of \cite{17, 10, 11} could
         expect the feedback laws  stabilizing systems with bigger sampling periods. As payments,
        the cost of using  such  feedback laws is  higher than that of using time-invariant feedback laws.

        For other sampled-data control problems, we mention here \cite{15} and \cite{03}.

        \noindent {\bf Motivations}. First, from what   mentioned in the
        previous works, we see that most time-invariant output feedback law can only stabilize systems when sampling period is sufficiently small, while the cost of using time-period output feedback laws is much higher than  that of using time-invariant output feedback laws.
         However,  the requirement that  sampling periods are sufficient small is not reasonable from two perspectives: smaller sampling periods require faster,  newer and more expensive hardware; performing all  control computations might not be feasible, when the sampling period is too small (see Section 1.2 in \cite{37}).  These motivate  us   to design time-invariant
          output feedback laws stabilizing sampled-data controlled systems with any  sampling period.

        Second,  it should be  important and interesting to study
        how  sampling periods effect the performance of systems. Indeed, such issue
                 has been extensively studied for the linear regulator problem of finite-dimensional systems (see \cite{33, 34, 35, 36}).
                 However, to our best knowledge, it has not been touched upon how sampling periods effect
                  feedback laws. This motivates us to study
                  the behaviors of the norm of the feedback law when the sampling period tends to zero or
                  infinity.

\subsection{ Main results}
The main results of our work will be precisely presented in Section 4. Here we state them  in plain language.

(i) For any $\gamma>0$ and $T>0$, we can design a feedback law $\mathcal{F}_T: L^2(\omega_1)\rightarrow L^2(\omega)$ so that any solution $y(\cdot)$ to the closed-loop equation:
 \begin{equation}\label{e102}
 \ \ \ \ \  \left\{\begin{array}{ll}
 y'(t)- Ay(t)=\mathds{1}_\omega \sum_{i=0}^{\infty}\chi_{[(2i+1)T, (2i+\frac{3}{2})T)}(t)\mathcal{F}_T(\mathds{1}^*_{\omega_1}y((2i+\frac{3}{4})T)),  \ t>0,\\
  y(0)\in L^2(\Omega),
\end{array}\right.
\end{equation}
satisfies that
\begin{equation}\label{e103}
\|y(t)\|\leq C e^{-\gamma t}\|y(0)\|, \;\; t>0,
\end{equation}
for some positive constant $C=C(T, \gamma, \gamma_0, \Omega, \omega, \omega_1, d)$. (Though $\mathcal{F}_T$ depends on both $\gamma$ and $T$,  we only care  about its dependence on $T$ in this study).

(ii) The feedback law is given in the following manner: First, we can find natural numbers  $M=M(T, \gamma, \gamma_0, \Omega, \omega, \omega_1, d)$ and $N=N(T, \gamma, \gamma_0, \Omega, \omega, \omega_1,d)$, which have  explicit expressions
in terms of $T$, $\gamma$ and $\gamma_0$. With the aid of such expressions, we can estimate $M$ and $N$ in terms of $T$, $\gamma$ and $\gamma_0$;
 Second, we can  find, for each $1\leq j\leq N$, two vectors $f_j\in \mathrm{span} \{\mathds{1}_\omega\xi_i\}_{i=1}^M$ and $h_j\in \mathrm{span} \{\mathds{1}_{\omega_1}\xi_i\}_{i=1}^M$, which have explicit formulas respectively in terms of the  unique minimizers of two explicitly expressed functionals on two $M$-dim spaces; Third, we can write the feedback law $\mathcal{F}_T$ as:
\begin{equation}\label{e104}
\mathcal{F}_T(v)=-\sum_{j=1}^N\langle v, h_j\rangle f_j, \;\; v\in L^2(\omega_1).
\end{equation}

(iii) We have  the following  estimates on  the feedback law:
\begin{equation}\label{e105}
\alpha(T)\leq \|\mathcal{F}_T\|_{\mathcal{L}(L^2(\omega_1); L^2(\omega))}\leq e^{C_1(1+\frac{1}{T})+\frac{\gamma_0T}{4}}e^{C_2(1+\frac{1}{T})\gamma}.
\end{equation}
Here,  $C_1=C_1(\gamma_0, \Omega, \omega, \omega_1, d)>0$ and $C_2=C_2(\gamma_0, \Omega, \omega, \omega_1, d)>0$, and  the function $\alpha(\cdot): (0, +\infty)\rightarrow (0, +\infty)$ satisfies that
\begin{eqnarray}\label{e106}
\alpha(T)=O({1}/{T})\ \mathrm{as}\ T\rightarrow 0; \ \ \ \
\alpha(T)=O(e^{-\frac{\lambda_1T}{4}})\ \mathrm{as}\ T\rightarrow +\infty.
\end{eqnarray}
From these, we see that
\begin{eqnarray}\label{puchongWANG1.10}
\lim_{T\rightarrow 0}\|\mathcal{F}_{T}\|_{\mathcal{L}(L^2(\omega_1); L^2(\omega))}=+\infty
\;\;\mbox{and}\;\;\lim_{T\rightarrow +\infty}\|\mathcal{F}_{T}\|_{\mathcal{L}(L^2(\omega_1); L^2(\omega))}=+\infty.
\end{eqnarray}

About the main results, several notes are given in order.
\begin{itemize}
\item Comparing with feedback laws built up in \cite{09, 14, 07, 13, 28} where the indirect method was applied,
    the advantage of our feedback law is that  the sampling period $2T$  can be any  positive number.
    Moreover, our feedback law has an explicit expression (\ref{e104}).

     \item The above  (iii) provides the following  useful information on $\mathcal{F}_T$:
     $(a)$ It gives both upper and lower bounds for the norm of the feedback law in terms of
     time period. $(b)$
     The cost of the feedback law
      blows up when $T$ goes to either zero or infinity. This guides
       us  to select the sampling period for our feedback law in such a way that  we better take the sampling period neither too small nor  too large.

      \item The second equality in (\ref{puchongWANG1.10})
      can be explained as:
      when using a piecewise function to approach a function over $[0,+\infty)$, the sampling period is larger, the error is bigger.
       However, the first equality in (\ref{puchongWANG1.10}) seems not to be  reasonable. This might be a weakness of our feedback law. One reason  causing it might be as:  our feedback law works for any  sampling period. Because its scope of application is too wide, it is imperfect in some places.
      Another reason causing it may be from our design of the feedback law.
      The key to build our feedback law is the use of
     the following controllability:  for any $\zeta\in L^2(\Omega)$, $\varepsilon>0$ and $T>0$, there is a control $u\in L^2(\omega)$
     so that
     the equation:  $z'(t)-Az(t)=\mathds{1}_\omega u, t\in(0, T/2)$ has  such a solution $z(\cdot)$ that $z(0)=\zeta$ and $\|P_Mz(T/2)\|\leq \varepsilon \|\zeta\|$. {\it Moreover, the cost of such a control is of order $e^{\frac{C}{T}}$
     in term of $T$.} (In the next section of this paper, we will show such controllability  by building up a new type of observability inequality.)
      The cost of the null controllability and the approximate controllability for heat equations have been extensively studied (see \cite{18, 19, 21}). It is acknowledged that the order $e^{\frac{C}{T}}$ for the upper bound of the cost for controls is sharp in some sense (see \cite{21}). From this perspective,  it is not surprise  that the norm of our feedback law goes to infinity when $T$ goes to zero.

  \item From the right hand side of (\ref{e102}), we can  see that our feedback control does not need to follow the rule:
      adding control with a feedback form simultaneously after taking sample. Instead of this,
      we take sample at time $(2i+\frac{3}{4})T$, and after a time delay, we add a feedback control
       over $[(2i+1)T, (2i+\frac{3}{2})T)$. Sometimes, such time delay could be useful in practical applications.

  \item  Our feedback law is constructed for any decay rate  $\gamma>0$. Such kind of stabilization is called the rapid stabilization. About this subject, we would like to mention the works \cite{29, 30, 31}.

\item  In the design of our feedback law, we borrowed some ideas from   \cite{01}, where the output feedback stabilization  for impulse controlled
    heat equations was studied.

\end{itemize}

\subsection{ Plan of this paper} The rest of this paper is organized as follows: Section 2 gives  a type of observability inequality and the partial null approximate controllability for  heat equations with time-invariant controls. Section 3 studies the norm optimal control problems governed by  heat equations with time-invariant controls and with
impulse controls respectively.
The  last section presents and proves   our main results.

\section{Partial null approximate controllability with a cost}
\setcounter{equation}{0}
Throughout this section, we arbitrarily fix $T_1$ and $T_2$, with $0\leq T_1<T_2$.
Given  $\zeta\in L^2(\Omega)$ and $u\in L^2(\omega)$, write  $y(\cdot;\zeta,u)$ for the solution to  the following  equation:
\begin{equation}\label{e201}
 \left\{\begin{array}{ll}
 y'(t)- A y(t)=\mathds{1}_\omega u,
 \   \ \ t\in (T_1 ,T_2),\\
  y(T_1)=\zeta.
\end{array}\right.
\end{equation}
\begin{definition}\label{def201}
The equation (\ref{e201}) is said to have the partially null approximate controllability with a cost, if for any
$M\in \mathds{N}^+$ and $\varepsilon >0$, there is  $C(\varepsilon, M, T_2-T_1)>0$
so that for each $\zeta\in L^2(\Omega)$,
 there is   $u^{\zeta}\in L^2(\omega)$ satisfying that
\begin{equation}\label{e202}
\frac{1}{[C(\varepsilon, M, T_2-T_1)]^2}(T_2-T_1)\|u^{\zeta}\|^2_{\omega}+\frac{1}{\varepsilon^2}\|P_My(T_2;\zeta,u^{\zeta})\|^2\leq \|\zeta\|^2.
\end{equation}
\end{definition}
\begin{remark}\label{rem201}
(i) In Definition \ref{def201},  $C(\varepsilon, M, T_2-T_1)$
can be treated as a cost for the control  $u^\zeta$, when $\zeta$ is an unit vector.
The reason is as: the control $u^\zeta$, when replaces $u$ on the right hand side of (\ref{e201}),
 is treated as a function: $t\rightarrow \chi_{(T_1, T_2)}(t)\mathds{1}_\omega u^\zeta$, $t\in(T_1,T_2)$. The $L^2(T_1, T_2; L^2(\omega))$-norm of this function is $\sqrt{T_2-T_1}\|u^{\zeta}\|_{\omega}$.

 (ii) Studies of the above-mentioned controllability is related to studies of the following adjoint equation:
\begin{equation}\label{e211}
 \left\{\begin{array}{ll}
 \varphi'(t)+A \varphi(t)=0,
 \   \ \ \ \ t\in (T_1 ,T_2),\\
 \varphi(T_2)=\Phi,
 \end{array}\right.
\end{equation}
where $\Phi\in L^2(\Omega)$. We write
 $\varphi(\cdot; T_2, \Phi)$ for the solution to the  equation
 (\ref{e211}).
 \end{remark}
For each $k\in \mathbb{N}^+$, we define a function $f_k: \mathds{R}^{k}\rightarrow \mathds{R}$ in the following manner:
\begin{equation}\label{203}
f_k(d_1, d_2, \dots, d_k)=\frac{|\langle \sum_{i=1}^{k}d_i\mathds{1}^*_\omega\xi_i, \mathds{1}^*_\omega\xi_{k+1}\rangle_\omega|}{\|\mathds{1}^*_\omega\xi_{k+1}\|_\omega},\;\;
(d_1,d_2,\dots, d_k)\in \mathbf{B}_k^\omega,
\end{equation}
where
\begin{equation}\label{e204}
\mathbf{B}_k^\omega=\Big\{(d_1, d_2, \cdots, d_k)\in \mathds{R}^k\; :\;\Big\|\sum_{i=1}^{k}d_i\mathds{1}^*_\omega\xi_i \Big\|_\omega=1\Big\}.
\end{equation}
It will be proved later that $f_k$ takes its maximum over $\mathbf{B}_k^\omega$ (see Lemma \ref{lem201}). We write
\begin{equation}\label{wang2.5}
\theta_k:=\max_{\mathbf{B}_k^\omega}f_k,\;\; k\in \mathbb{N}^+.
\end{equation}
For each $M\in \mathbb{N}^+$, we let
\begin{equation}\label{e206}
\tau_1:=1\;\;\mbox{and}\;\;\tau_M :=\sqrt{\frac{M}{\prod_{k=1}^{M-1}(1-\theta_{k})}},\ M\geq 2.
\end{equation}
 The main result of this section is as:
\begin{theorem}\label{th201}
The equation (\ref{e201})
has the partially null approximate controllability with a cost, i.e.,  given $ \zeta \in L^2(\Omega)$, $M\in \mathds{N}^+$ and $\varepsilon>0$, there is $\tilde{u}\in L^2(\omega)$ so that
\begin{equation}\label{e207}
\|P_My(T_2;\zeta,\tilde{u})\|\leq \varepsilon\|\zeta\|;
\end{equation}
and
\begin{equation}\label{e208}
\sqrt{T_2-T_1}\|\tilde{u}\|_\omega\leq C(\varepsilon, M, T_2-T_1)\|\zeta\|.
\end{equation}
Here
\begin{equation}\label{e209}
C(\varepsilon, M, T_2-T_1)=\tau_Me^{C\big(1+\frac{1}{T_2-T_1}+(T_2-T_1)\gamma_0
+\gamma_0^{\frac{2}{3}}+\sqrt{\frac{1}{T_2-T_1} \ln^+\frac{1}{\varepsilon}}\big)},
\end{equation}
where  $\tau_M$ and $\gamma_0$ are given by (\ref{e206}) and (\ref{wang1.2}) respectively, and where  $C=C(\omega, \Omega)$.
\end{theorem}

To prove Theorem \ref{th201}, we need next two lemmas: Lemma \ref{lem202} and Lemma  \ref{lem203}. The proof of Lemma \ref{lem202} needs the following Lemma \ref{lem201}:
(The proofs of these three lemmas will be given after  the proof of Theorem \ref{th201}.)
\begin{lemma}\label{lem201}
For each  $k\in\mathbb{N}^+$, the function $f_k$ (defined by (\ref{203}), as well as
(\ref{e204})) takes its maximum $\theta_k$ at some point in $\mathbf{B}_k^\omega$. Moreover, it holds that
\begin{equation}\label{e205}
0\leq\theta_k<1.
\end{equation}
\end{lemma}

\begin{lemma}\label{lem202}
Let $\chi_\omega$ be the characteristic function of $\omega$.
 The following statements are equivalent and are true.\\
 (i) There is $C=C(\Omega, \omega)>0$ so that for any $\theta\in (0,1)$, $M\in \mathds{N}^+$ and  $\Phi\in X_M$,
\begin{eqnarray}\label{e210}
 \|\varphi(T_1; T_2, \Phi)\|
&\leq& \tau_M^{1-\theta}e^{C\big(1+\frac{1}{\theta(T_2-T_1)}+(T_2-T_1)\gamma_0
+\gamma_0^{\frac{2}{3}}\big)}
\|\Phi\|^\theta
\nonumber\\
 &&\cdot\Big\|\chi_\omega\frac{1}{\sqrt{T_2-T_1}}\int_{T_1}^{T_2}\varphi(t; T_2, \Phi)dt\Big\|^{1-\theta},
\end{eqnarray}
where $\tau_M$ is given by (\ref{e206}) and $\varphi(\cdot; T_2, \Phi)$ is the solution to (\ref{e211}).
\\
(ii) For any $\varepsilon>0$, $M\in \mathds{N}^+$ and  $\Phi\in X_M$,
\begin{eqnarray}\label{e212}
&&\|\varphi(T_1; T_2, \Phi)\|^2\nonumber\\
&\leq& \varepsilon^2\|\Phi\|^2
+[C(\varepsilon, M, T_2-T_1)]^2\Big\|\chi_\omega\frac{1}{\sqrt{T_2-T_1}}\int_{T_1}^{T_2}\varphi(t; T_2, \Phi)dt\Big\|^2.
\end{eqnarray}
Here $C(\varepsilon, M, T_2-T_1)$ is given by (\ref{e209})
and $\varphi(\cdot; T_2, \Phi)$ is the solution to (\ref{e211}).
\end{lemma}

\begin{lemma}\label{lem203}
 Equation (\ref{e201}) has the  partially null approximate controllability with a cost if and only if for any $\varepsilon>0$ and $M\in \mathds{N}^+$, there is $C(\varepsilon, M, T_2-T_1)$ so that for any $\Phi\in X_M$,
 the solution $\varphi(\cdot; T_2,\Phi)$ to (\ref{e211}) satisfying that
\begin{eqnarray}\label{e213}
&&\|\varphi(T_1; T_2, \Phi)\|^2\nonumber\\
&\leq& \varepsilon^2\|\Phi\|^2
+[C(\varepsilon, M, T_2-T_1)]^2\Big\|\chi_\omega\frac{1}{\sqrt{T_2-T_1}}\int_{T_1}^{T_2}\varphi(t; T_2, \Phi)dt\Big\|^2.
\end{eqnarray}
Moreover, the constant $C(\varepsilon, M, T_2-T_1)$
in (\ref{e213}) and the  constant $C(\varepsilon, M, T_2-T_1)$ in (\ref{e202}) are the same.
\end{lemma}
\begin{remark}\label{rem202}
(i) Given $\theta\in (0,1)$ and $M\in \mathbb{N}^+$,  the inequality (\ref{e210}) does not hold for all $\Phi \in L^2(\Omega)$. This can be proved as follows: Define
$\mathcal{S}: L^2(\Omega)\rightarrow H^1_0(\Omega)\cap H^2(\Omega)$ in the following manner: $\mathcal{S}(\Phi):=\int_{T_1}^{T_2}\varphi(t; T_2,\Phi) dt$, $\;\Phi\in L^2(\Omega)$.
By the properties of the $C_0$-semigroup (see, for instance, Theorem 2.4 in Chapter 1 of \cite{PAZY}),   $\mathcal{S}$ is well defined, i.e.,
$\int_{T_1}^{T_2}\varphi(t; T_2,\Phi) dt\in D(A)=H^1_0(\Omega)\cap H^2(\Omega)$ for each $\Phi\in L^2(\Omega)$. By making use of the
 Galerkin method, we can directly show that $\mathcal{S}$  is bijective. (We omit the detailed proof.) Next, we take a nonzero function $\Psi\in H^1_0(\Omega)\cap H^2(\Omega)$ so that $\Psi=0$ in $\omega$. Since $\mathcal{S}$ is bijective, we have that   $\widetilde{\Phi}=\mathcal{S}^{-1}\Psi\neq 0$ in $L^2(\Omega)$.
 This, along with the backward uniqueness of heat equations, yields that
 $ e^{A(T_2-T_1)}\widetilde{\Phi}\neq 0$, i.e.,
 $\varphi(T_1;T_2,\widetilde{\Phi})\neq 0$.
  Thus, for this $\widetilde{\Phi}$, the left hand side of (\ref{e210}) is
 $\|\varphi(T_1;T_2,\widetilde{\Phi})\|$ which is positive, while the right hand side of (\ref{e210})
 is zero, since
 $ \chi_\omega\int_{T_1}^{T_2}\varphi(T_1;T_2,\widetilde{\Phi})
 =\chi_\omega\mathcal{S}(\widetilde{\Phi})=\chi_\omega\psi=0$.
 Hence,  the function $\widetilde{\Phi}$ does not satisfy the inequality (\ref{e210}).

(ii) The  equation (\ref{e201}) is not null approximate controllable,
i.e., there is $\zeta\in L^2(\Omega)$ and $\varepsilon>0$ so that for any $u\in L^2(\omega)$, $\|y(T_2;\zeta,u)\|\geq \varepsilon$. Here is the proof: Define
$\mathcal{S}_\omega: L^2(\Omega)\rightarrow L^2(\omega)$ in the following manner:
$
\mathcal{S}_\omega(\Phi):=\mathds{1}_\omega^*\int_{T_1}^{T_2}\varphi(t; T_2,\Phi) dt$, $\;
\Phi\in L^2(\Omega)$.
Its adjoint operator reads:
$\mathcal{S}^*_\omega(u)=\int_{T_1}^{T_2}e^{A(T_2-t)}\mathds{1}_\omega udt$, $\; u\in L^2(\omega)$.
Since $\mbox{ker}(\mathcal{S}_\omega)\neq\{0\}$, (This has been proved in (i) of this remark.) we see that
$\mbox{range}(\mathcal{S}^*_\omega)$ is not dense in $L^2(\Omega)$. So there is
 $\zeta_1\in L^2(\Omega)$ and $\varepsilon>0$ so that
 \begin{equation}\label{wanggs2.15}
 \|\mathcal{S}^*_\omega(u)-\zeta_1\|>2\varepsilon\;\;\mbox{for all}\;\;u\in L^2(\omega).
 \end{equation}
 Meanwhile, since  the set $\{e^{A(T_2-T_1)}\zeta:\; \zeta\in L^2(\Omega)\}$ is dense in $L^2(\Omega)$, we can choose $\zeta\in L^2(\Omega)$ so that
 $ \|e^{A(T_2-T_1)}\zeta-\zeta_1\|< \varepsilon$.
 This, along with (\ref{wanggs2.15}), yields that for any
   $u\in L^2(\omega)$,
$$
\|y(T_2; \zeta, u)\|=\|e^{A(T_2-T_1)}\zeta
+\mathcal{S}^*_\omega(u)\|\geq\|\mathcal{S}^*_\omega(u)-\zeta_1\|
-\|e^{A(T_2-T_1)}\zeta-\zeta_1\|>\varepsilon.
$$

(iii) A similar observability inequality to (\ref{e210}) was obtained in \cite{03}
(see Lemma 2.3 in \cite{03}):  There exists $C=C(\Omega,\omega)$ so that for any
$T$, $S$, with $0<S<T$,
\begin{equation}\label{e214}
\ \ \ \ \ \|\varphi(0; T, \Phi)\|\leq e^{C(1+\frac{1}{T-S})}
\|\Phi\|^{\frac{1}{2}}\Big\|\chi_\omega\frac{1}{S}\int_{0}^S\varphi(t; T,\Phi)dt\Big\|^{\frac{1}{2}}\;\;\mbox{for all}\;\;\Phi\in L^2(\Omega),
\end{equation}
where $\varphi(\cdot; T, \Phi)$ solves  the equation:
\begin{equation}\label{e215}
 \left\{\begin{array}{ll}
 \varphi'(t)+\Delta \varphi(t)=0,
 \   \ \ \ \ t\in (0 ,T),\\
 \varphi(T)=\Phi.
\end{array}\right.
\end{equation}
The main differences between (\ref{e210}) and  (\ref{e214}) are as follows: First,
(\ref{e214}) holds for all $\Phi\in L^2(\Omega)$, while (\ref{e210}) holds  for all $\Phi\in X_M$, but not all $\Phi\in L^2(\Omega)$. Second, the integral interval on the righthand side of (\ref{e214}) cannot be the whole interval $[0, T]$, while it can be in (\ref{e210}).

\end{remark}

We now on the position to prove  Theorem \ref{th201}.

\vskip 5pt

\noindent \emph{Proof of Theorem \ref{th201}.}
From Lemma \ref{lem202} and Lemma \ref{lem203}, we see
that for any $\zeta\in L^2(\Omega)$, $M\in \mathbb{N}^+$ and $\varepsilon>0$,
there is $\tilde{u}\in L^2(\omega)$, with (\ref{e208}) where $C(\varepsilon, M, T_2-T_1)$ is given by (\ref{e209}), so that (\ref{e207}) holds. This ends the proof
of Theorem \ref{th201}.
$\hfill\square$

We next prove  Lemmas \ref{lem201},  \ref{lem202}, \ref{lem203} one by one.

\vskip 5pt

\noindent\emph{Proof of Lemma \ref{lem201}.}
Arbitrarily fix $k\in \mathbb{N}^+$. By the unique continuation property for the elliptic equations (see Theorem 15.2.1 in \cite{32}), we find that $\mathds{1}^*_\omega\xi_{k+1}$ is not zero, from which, it follows that $f_k$ is well defined. Moreover,
from (\ref{203}), we can easily check that $f_k$ is continuous.

We now claim that $\mathbf{B}_k^\omega$ is closed and bounded in $\mathbb{R}^k$.
Indeed, the closedness of $\mathbf{B}_k^\omega$ can be easily verified. To show  its boundedness, we arbitrarily fix $(d_1, \cdots, d_k)\in \mathbf{B}_k^\omega$. Write
$\psi:=\sum_{i=1}^{k}d_i\mathds{1}^*_\omega\xi_i$.
  It follow from (\ref{e204}) that $\|\psi\|_\omega=1$. Let
 \begin{equation}\label{wangggs2.19}
 B_k:=\left(
    \begin{array}{ccc}
      \langle\mathds{1}^*_\omega\xi_1,
      \mathds{1}^*_\omega\xi_1 \rangle_\omega & \cdots & \langle\mathds{1}^*_\omega\xi_1,\mathds{1}^*_\omega\xi_k \rangle_\omega \\
      \cdots  & \cdots  & \cdots  \\
      \langle\mathds{1}^*_\omega\xi_k,\mathds{1}^*_\omega\xi_1 \rangle_\omega  & \cdots  &  \langle\mathds{1}^*_\omega\xi_k,\mathds{1}^*_\omega\xi_k \rangle_\omega  \\
    \end{array}
  \right).
  \end{equation}
   Since $\mathds{1}^*_\omega\xi_1, \dots,\mathds{1}^*_\omega\xi_k$ are linearly independent in $L^2(\omega)$
 (see, for instance, Page 38, Section 2.2 in \cite{22}),
 the matrix
 $B_k$ (given by (\ref{wangggs2.19})) is invertible.
 Moreover,  one can easily check  that
  $$
  (d_1, \dots, d_k)'=B_k^{-1}(\langle\psi,\mathds{1}^*_\omega\xi_1 \rangle_\omega , \dots,  \langle\psi,\mathds{1}^*_\omega\xi_k \rangle_\omega )'.
  $$
This, along with the fact that  $\|\psi\|_\omega=1$, yields that $\mathbf{B}_k^\omega$ is bounded. Hence, the continuous function $f_k$ takes its maximum on the closed and bounded subset $\mathbf{B}_k^\omega$ in $\mathbb{R}^k$.

We next show (\ref{e205}).
From (\ref{203}) and (\ref{e204}), we find that $0\leq\theta_k\leq 1$.
By contradiction, we suppose that
 $\theta_k=1$. Then there would be $(d_1, \dots, d_k)\in \mathbf{B}_k^\omega$ so that
$$
\Big|\langle \sum_{i=1}^{k}d_i\mathds{1}^*_\omega\xi_i, \mathds{1}^*_\omega\xi_{k+1}\rangle_\omega\Big|
=\left\|\mathds{1}^*_\omega\xi_{k+1}\right\|_\omega.
$$
Meanwhile, it is clear that
$$
\Big|\Big\langle \sum_{i=1}^{k}d_i\mathds{1}^*_\omega\xi_i, \mathds{1}^*_\omega\xi_{k+1}\Big\rangle_\omega\Big|\leq \Big\|\sum_{i=1}^{k}d_i\mathds{1}^*_\omega\xi_i\Big\|_\omega
\|\mathds{1}^*_\omega\xi_{k+1}\|_\omega=\|\mathds{1}^*_\omega\xi_{k+1}\|_\omega.
$$
From these, we can find $\alpha\neq 0$ so that
$\alpha\sum_{i=1}^{k}d_i\mathds{1}^*_\omega\xi_i=\mathds{1}^*_\omega\xi_{k+1}$.
This contradicts to the linear independence of $\mathds{1}^*_\omega\xi_1, \dots,\mathds{1}^*_\omega\xi_{k+1}$.
Hence, (\ref{e205}) is true.
This ends the proof of Lemma \ref{lem201}.
$\hfill\square$

\vskip 7pt

\noindent\emph{Proof of Lemma \ref{lem202}.}
The proof is organized by two  parts.

\noindent \emph{Part I. We prove that statements  (i) and  (ii) are equivalent.}

To show that
(i)$\Rightarrow$ (ii), we arbitrarily fix $\varepsilon>0$, $M\in \mathbb{N}^+$ and $\Phi\in X_M$.
Recall that $\tau_M$ is given by (\ref{e206}).
It follows from (\ref{e210}) that for any $\theta\in (0,1)$,
\begin{equation}\ \ \ \ \ \label{e217}
\|\varphi(T_1; T_2, \Phi)\|^2\leq (\|\Phi\|^2)^\theta \left(\tau_M^{2}e^{\frac{2C}{1-\theta}\big(1+\frac{1}{\theta(T_2-T_1)}
+(T_2-T_1)\gamma_0
+\gamma_0^{\frac{2}{3}}\big)}\|\chi_\omega\bar{\varphi}\|^2\right)^{1-\theta},
\end{equation}
where $\bar{\varphi}=\frac{1}{\sqrt{T_2-T_1}}\int_{T_1}^{T_2}\varphi(t; T_2, \Phi)dt$. Write
$$
\beta=\frac{\theta}{1-\theta},\; \Lambda=C(1+\frac{1}{T_2-T_1}+(T_2-T_1)\gamma_0+\gamma_0^{\frac{2}{3}}), \; \Upsilon=\frac{C}{T_2-T_1}.
$$
 Applying  Young's inequality to (\ref{e217}) yields that
\begin{eqnarray}\label{e218}
&&\|\varphi(T_1; T_2, \Phi)\|^2\nonumber\\
&\leq&\varepsilon^2\|\Phi\|^2
+(1-\theta)\theta^\beta\frac{1}{\varepsilon^{2\beta}}
\tau_M^{2}e^{2C(1+\beta)\big(1+\frac{1+\beta}{\beta(T_2-T_1)}
+(T_2-T_1)\gamma_0+\gamma_0^{\frac{2}{3}}\big)}\|\chi_\omega\bar{\varphi}\|^2\nonumber\\
&\leq&\varepsilon^2\|\Phi\|^2
+\tau_M^{2}e^{2\Lambda+2\Upsilon
+2\beta(\mathrm{ln}^+\frac{1}{\varepsilon}+\Lambda)
+\frac{2}{\beta}\Upsilon}\|\chi_\omega\bar{\varphi}\|^2.
\end{eqnarray}
By choosing $\theta$ so that  $\beta=\sqrt{\frac{\Upsilon}{\ln^+\frac{1}{\varepsilon}+\Lambda}}$, we obtain from
(\ref{e218}) that
\begin{equation}\label{e219}
\|\varphi(T_1; T_2, \Phi)\|^2\leq \tau_M^{2}e^{8\Lambda+4\sqrt{\Upsilon \ln^+\frac{1}{\varepsilon}}}\|\chi_\omega\bar{\varphi}\|^2+\varepsilon^2\|\Phi\|^2,
\end{equation}
which leads to  (\ref{e212}) with $C(\varepsilon, M, T_2-T_1)$  given by (\ref{e209}).

To show that (ii)$\Rightarrow$(i), we arbitrarily fix $\theta\in (0,1)$, $M\in \mathbb{N}^+$ and $\Phi\in X_M$.
Write  $C=C(\omega,\Omega)$ for a  constant depending only on $\Omega$ and $\omega$, which may vary in different contexts.
 One can directly check that
$$
2C\sqrt{\frac{1}{T_2-T_1} \ln^+\frac{1}{\varepsilon}}\leq \frac{C^2}{2\alpha(T_2-T_1)}+ 2\alpha\ln\Big(e+\frac{1}{\varepsilon}\Big)
\;\;\mbox{for all}\;\;\varepsilon>0, \alpha>0.
$$
From this and  (\ref{e212}), we find  that for any $\varepsilon>0$ and $\alpha>0$,
\begin{eqnarray*}
\|\varphi(T_1; T_2, \Phi)\|^2
\leq \varepsilon^2\|\Phi\|^2
+\tau_M^2e^{2C\big(1+\frac{1}{T_2-T_1}+(T_2-T_1)\gamma_0
+\gamma_0^{\frac{2}{3}}\big)}e^{\frac{C^2}{2\alpha(T_2-T_1)}}
(e+{1}/{\varepsilon})^{2\alpha}\|\chi_\omega\overline{\varphi}\|^2.
\end{eqnarray*}
We may as well assume that $\Phi\neq 0$. Taking  $\varepsilon=\frac{1}{2}\frac{\|\varphi(T_1; T_2, \Phi)\|}{\|\Phi\|}$
 in  the above, and then using the inequality: $\|\varphi(T_1; T_2, \Phi)\|\leq e^{\gamma_0(T_2-T_1)}\|\Phi\|$, we see that for any $\alpha>0$,
\begin{eqnarray*}
&&\frac{3}{4}\|\varphi(T_1; T_2, \Phi)\|^{2(1+\alpha)}
\\ \nonumber
&\leq &\tau_M^2e^{2C\big(1+\frac{1}{T_2-T_1}+(T_2-T_1)\gamma_0
+\gamma_0^{\frac{2}{3}}\big)}e^{\frac{C^2}{2\alpha(T_2-T_1)}}
(e^{\gamma_0(T_2-T_1)+1}+2)^{2\alpha}\|\Phi\|^{2\alpha}\|
\chi_\omega\overline{\varphi}\|^2.
\end{eqnarray*}
 Taking $\alpha>0$ so that  $\theta=\frac{\alpha}{1+\alpha}$, we can obtain from
 the above that
 \begin{eqnarray*}
\|\varphi(T_1; T_2, \Phi)\|
\leq \tau_M^{1-\theta}e^{{C}\big(1+\frac{1}{\theta(T_2-T_1)}
+(T_2-T_1)\gamma_0
+\gamma_0^{\frac{2}{3}}\big)}\|
\Phi\|^{\theta}\|\chi_\omega\overline{\varphi}\|^{1-\theta},
\end{eqnarray*}
which leads to (\ref{e210}).

\noindent\emph{Part II. We show that the proposition (i) is true.}

Arbitrarily fix $\theta\in (0,1)$, $M\in\mathbb{N}^+$ and $\Phi\in X_M$.
Write
\begin{equation}\label{wanggs2.25}
\Phi=\sum_{i=1}^Ma_i\xi_i,\;\;\mbox{with}\;\;a_i\in \mathbb{R}.
\end{equation}
The proof of (i) is divided into the following two steps:

\noindent\emph{Step 1. We  prove that there exists $C=C(\omega, \Omega)$ so that
\begin{eqnarray}\label{e223}
\ \ \ \ \ &&\ \ \|\varphi(T_1; T_2, \Phi)\|\nonumber\\
&\ \ \ \leq & e^{C\big(1+\frac{1}{\theta(T_2-T_1)}+(T_2-T_1)\gamma_0+\gamma_0^{\frac{2}{3}}\big)}
\|\Phi\|^\theta\Big\|\chi_\omega
\frac{1}{\sqrt{T_2-T_1}}\int_{T_1}^{\frac{T_1+T_2}{2}}\varphi(t; T_2, \Phi )dt\Big\|^{1-\theta}.
\end{eqnarray}}

Define
\begin{equation}\label{wanggs2.27}
\eta:=\frac{1}{\sqrt{T_2-T_1}}\int_{T_1}^{\frac{T_1+T_2}{2}}e^{A(\frac{T_1+T_2}{2}-t)}\Phi dt.
\end{equation}
Then
\begin{equation}\label{e225}
\|\eta\|=\Big\|\frac{1}{\sqrt{T_2-T_1}}
\int_{T_1}^{\frac{T_1+T_2}{2}}e^{A(\frac{T_1+T_2}{2}-t)}\Phi dt\Big\|\leq \frac{\sqrt{T_2-T_1}}{2}e^{\gamma_0\frac{T_2-T_1}{2}}\|\Phi\|.
\end{equation}
Moreover, it follows  by (\ref{wanggs2.27}) and (\ref{wanggs2.25}) that
\begin{eqnarray}\label{e226}
\left\|e^{\frac{T_2-T_1}{2}A}\eta\right\|
&=&\frac{1}{\sqrt{T_2-T_1}}\Big[\sum_{i=1}^M\Big(\int_{T_1}^{\frac{T_1
+T_2}{2}}e^{\lambda_i(t-T_1)}dt\Big)^2
(e^{-\lambda_i(T_2-T_1)}a_i)^2\Big]^{\frac{1}{2}}\nonumber\\
&\geq&\frac{\sqrt{T_2-T_1}}{2}e^{-\gamma_0\frac{T_2-T_1}{2}}\left\|\varphi(T_1; T_2, \Phi)\right\|.
\end{eqnarray}
Meanwhile, it follows by (iii) of Theorem 2.1 in \cite{01} (see also \cite{20}) that
\begin{equation}\label{e224}
\|e^{\frac{T_2-T_1}{2}A}\eta\|\leq e^{C\big(1+\frac{1}{\theta(T_2-T_1)}+(T_2-T_1)\gamma_0+\gamma_0^{\frac{2}{3}}\big)}
\|\eta\|^{\theta}\|\chi_\omega e^{\frac{T_2-T_1}{2}A}\eta\|^{1-\theta}.
\end{equation}
Then  from (\ref{e226}), (\ref{e224}) and (\ref{e225}), we see that
\begin{eqnarray*}
\|\varphi(T_1; T_2, \Phi)\|\leq e^{C\big(1+\frac{1}{\theta(T_2-T_1)}+(T_2-T_1)\gamma_0+\gamma_0^{\frac{2}{3}}\big)}
\|\Phi\|^\theta\Big\|\chi_\omega\frac{1}{T_2-T_1}\int_{T_1}^{\frac{T_1
+T_2}{2}}\varphi(t; T_2, \Phi )dt\Big\|^{1-\theta}\\
\leq e^{C\big(1+\frac{1}{\theta(T_2-T_1)}+(T_2-T_1)\gamma_0+\gamma_0^{\frac{2}{3}}\big)}
\|\Phi\|^\theta\Big\|\chi_\omega\frac{1}{\sqrt{T_2-T_1}}\int_{T_1}^{\frac{T_1
+T_2}{2}}\varphi(t; T_2, \Phi )dt\Big\|^{1-\theta},
\end{eqnarray*}
which leads to (\ref{e223}).

\noindent\emph{Step 2. We show that
\begin{equation}\label{e228}
\ \ \ \ \ \ \Big\|\chi_\omega\frac{1}{\sqrt{T_2-T_1}}\int_{T_1}^{\frac{T_1+T_2}{2}}\varphi(t; T_2, \Phi )dt\Big\|\leq \tau_M\Big\|\chi_\omega\frac{1}{\sqrt{T_2-T_1}}\int_{T_1}^{T_2}\varphi(t; T_2, \Phi)dt\Big\|.
\end{equation}}
For each $1\leq j\leq M$, we let
$$
\alpha_j:=\int_{T_1}^{T_2}e^{-\lambda_j(T_2-t)}dt, \;\; \beta_j:=\int_{T_1}^{\frac{T_1+T_2}{2}}e^{-\lambda_j(T_2-t)}dt, \;\; \gamma_j:=\frac{\alpha_j}{\beta_j}, \;\; \eta_j:=\beta_ja_j\mathds{1}_\omega\mathds{1}^*_\omega\xi_j.
$$
Then, by (\ref{wanggs2.25}), we see that
\begin{eqnarray*}
\chi_\omega\int_{T_1}^{\frac{T_1+T_2}{2}}\varphi(t; T_2, \Phi )dt=\sum_{j=1}^M\int_{T_1}^{\frac{T_1+T_2}{2}}
e^{-\lambda_j(T_2-t)}dta_j\mathds{1}_\omega\mathds{1}^*_\omega\xi_j
=\sum_{j=1}^M\eta_j;\nonumber\\
\chi_\omega\int_{T_1}^{T_2}\varphi(t; T_2, \Phi)dt=\sum_{j=1}^M\int_{T_1}^{T_2}e^{-\lambda_j
(T_2-t)}dta_j\mathds{1}_\omega\mathds{1}^*_\omega\xi_j=\sum_{j=1}^M\gamma_j\eta_j.
\end{eqnarray*}
Hence, the inequality (\ref{e228}) is equivalent to the following inequality:
\begin{equation}\label{e229}
\|\sum_{j=1}^M\eta_j\|^2\leq \tau_M^2\|\sum_{j=1}^M\gamma_j\eta_j\|^2.
\end{equation}
 We may as well assume that $M\geq 2$. In order to show (\ref{e229}), we observe that
\begin{eqnarray}\label{e230}
\ \ \ \ &&\Big\|\sum_{j=1}^M\gamma_j\eta_j\Big\|^2
=\Big\|\sum_{j=1}^{M-1}\gamma_j\eta_j\Big\|^2
+2\gamma_M\Big\langle\sum_{j=1}^{M-1}\gamma_j\eta_j, \eta_M \Big\rangle+\gamma_M^2\|\eta_M\|^2\nonumber\\
&\geq&\Big\|\sum_{j=1}^{M-1}\gamma_j\eta_j\Big\|^2
-2\gamma_M\Big|\Big\langle\frac{\sum_{j=1}^{M-1}
\alpha_ja_j\mathds{1}^*_\omega\xi_j}{\|\sum_{j=1}^{M-1}
\alpha_ja_j\mathds{1}^*_\omega\xi_j\Big\|_\omega}, \frac{\mathds{1}^*_\omega\xi_{M}}
{\|\mathds{1}^*_\omega\xi_{M}\|_\omega} \Big\rangle_\omega\Big|
\Big\|\sum_{j=1}^{M-1}\gamma_j\eta_j\Big\|\|\eta_M\|\nonumber\\
&&+\gamma_M^2\|\eta_M\|^2.
\end{eqnarray}
Meanwhile,  we take
$\hat d_j=\frac{\alpha_ja_j}{\|\sum_{j=1}^{M-1}\alpha_ja_j\mathds{1}^*_\omega\xi_j\|_\omega}$,
with $j=1,\dots, M-1$.
One can easily see that $(\hat d_1,\dots, \hat d_{M-1})\in \mathbf{B}_{M-1}^\omega$
(see (\ref{e204})). Thus, we can apply  Lemma \ref{lem201} with $k=M-1$ to see that
$f_{M-1}(\hat d_1,\dots, \hat d_{M-1})\leq\theta_{M-1}$.
This, along with (\ref{e230}), indicates that
\begin{eqnarray}\label{e231}
&&\Big\|\sum_{j=1}^M\gamma_j\eta_j\Big\|^2\nonumber\\
&\geq&\Big\|\sum_{j=1}^{M-1}\gamma_j\eta_j\Big\|^2
-2\gamma_M\theta_{M-1}\Big\|\sum_{j=1}^{M-1}\gamma_j\eta_j\|\|\eta_M\Big\|
+\gamma_M^2\|\eta_M\|^2\nonumber\\
&\geq&(1-\theta_{M-1})\Big(\Big\|\sum_{j=1}^{M-1}\gamma_j\eta_j\Big\|^2
+\gamma_M^2\|\eta_M\|^2\Big).
\end{eqnarray}
By the similar way used in the proof of  (\ref{e231}), we can obtain that
\begin{equation}\label{e232}
\Big\|\sum_{j=1}^{M-1}\gamma_j\eta_j\Big\|^2\geq (1-\theta_{M-2})\Big(\Big\|\sum_{j=1}^{M-2}\gamma_j\eta_j\Big\|^2
+\gamma_{M-1}^2\|\eta_{M-1}\|^2\Big).
\end{equation}
Submitting (\ref{e232}) into (\ref{e231}), we see that
\begin{eqnarray}\label{e233}
&&\Big\|\sum_{j=1}^M\gamma_j\eta_j\Big\|^2
\geq (1-\theta_{M-1})(1-\theta_{M-2})
\Big\|\sum_{j=1}^{M-2}\gamma_j\eta_j\Big\|^2\nonumber\\
&&+(1-\theta_{M-1})(1-\theta_{M-2})\gamma_{M-1}^2\|\eta_{M-1}\|^2
+(1-\theta_{M-1})\gamma_M^2\|\eta_M\|^2.
\end{eqnarray}
Proceeding the above step by step, using facts:
 $\Big\|\sum_{j=1}^M\eta_j\Big\|^2\leq M\sum_{j=1}^M\|\eta_j\|^2$ and $\gamma_j>1$ for all $j=1, 2, \dots, M$,
    we find that
\begin{eqnarray*}
\Big\|\sum_{j=1}^M\gamma_j\eta_j\Big\|^2
&\geq&\gamma_1^2\prod_{k=1}^{M-1}(1-\theta_{M-k})\|\eta_1\|^2
+\sum_{j=2}^{M}\Big[\gamma_j^2\prod_{k=1}^{M-j+1}(1-\theta_{M-k})\|\eta_j\|^2\Big]\\
&\geq& \frac{1}{M}\prod_{k=1}^{M-1}(1-\theta_{M-k})\Big\|\sum_{j=1}^M\eta_j\Big\|^2,
\end{eqnarray*}
which leads to (\ref{e229}). Hence,  (\ref{e228}) is true. Then by (\ref{e228})
and  (\ref{e223}), we obtain  (\ref{e210}).

Thus, we end the proof of Lemma \ref{lem202}.
$\hfill\square$

\vskip 7pt

\noindent\emph{Proof of Lemma \ref{lem203}.} The proof is organized by the following three steps:

\noindent \emph{Step 1. We show that (\ref{e213}) implies  the partially null approximate controllability with a cost. }

 Suppose that  (\ref{e213}) holds. Arbitrarily fix $\varepsilon>0$, $M\in\mathbb{N}^+$
 and $\zeta\in L^2(\Omega)$. Let $C(\varepsilon, M, T_2-T_1)$ be given by (\ref{e213}).
 Then let
 \begin{equation}\label{wanggs2.37}
 \hbar:=\varepsilon^2\;\;\mbox{and}\;\; \kappa:=[C(\varepsilon, M, T_2-T_1)]^2.
  \end{equation}
  Define the  functional $F: X_M\rightarrow \mathds{R}$ via
$$
F(\Phi)=\frac{\kappa}{2}\Big\|\mathds{1}_\omega^*
\frac{1}{\sqrt{T_2-T_1}}\int_{T_1}^{T_2}e^{A(T_2-t)}\Phi dt\Big\|_\omega^2+\frac{\hbar}{2}\|\Phi\|^2-\langle \zeta, e^{(T_2-T_1)A}\Phi\rangle,\; \Phi\in X_M.
$$
One can easily check that
it is  strictly convex, coercive and of $C^1$. Thus it  has a unique minimizer $\bar{\Phi}\in X_M$ so that $F(\bar{\Phi})=\min_{\Phi\in X_M}F(\Phi)$.
The Euler-Lagrange equation of $F$ associated with $\bar{\Phi}$  is as:
\begin{eqnarray}\label{e234}
&&\kappa \Big\langle \mathds{1}_\omega^*\frac{1}{\sqrt{T_2-T_1}}\int_{T_1}^{T_2}e^{A(T_2-t)}\Phi dt, \; \mathds{1}_\omega^*\frac{1}{\sqrt{T_2-T_1}}\int_{T_1}^{T_2}e^{A(T_2-t)}\bar{\Phi}dt \Big\rangle_\omega\nonumber\\
&+&\hbar\langle \Phi, \bar{\Phi}\rangle
=\langle \zeta, \;e^{(T_2-T_1)A}\Phi\rangle\;\;\mbox{for all}\;\; \Phi\in X_M.
\end{eqnarray}
Define a time-invariant control:
\begin{equation}\label{wang2.39}
\tilde{u}:=-\kappa \frac{1}{T_2-T_1}\mathds{1}_\omega^*\int_{T_1}^{T_2}e^{A(T_2-t)}\bar{\Phi}dt.
\end{equation}

By (\ref{e234}) (where $\Phi=\bar{\Phi}$),
(\ref{wanggs2.37}) and
(\ref{e213}), we can use a standard way (see the proof of Theorem 3.1 in \cite{01}) to show that  $\tilde{u}$ satisfies  (\ref{e202}) with the above-mentioned $C(\varepsilon, M, T_2-T_1)$.
 So the equation (\ref{e201}) has the  partially null approximate controllability with a cost.

\noindent\emph{Step 2. We prove that the partial null approximate controllability with a cost implies (\ref{e213}). }

Suppose that the equation (\ref{e201}) has the partially null approximate controllability with a cost.
Arbitrarily fix $\varepsilon>0$, $M\in\mathbb{N}^+$ and $\Phi\in X_M$. Then
there is $C(\varepsilon, M, T_2-T_1)>0$ so that for any $\zeta\in L^2(\Omega)$,
there is $u^{\zeta}\in L^2(\omega)$ satisfying (\ref{e202}).
Multiplying  (\ref{e201}) (where $u=u^{\zeta}$)   by $\varphi(\cdot;;T_2,\Phi)$, and integrating on $[T_1, T_2]$, we see that
\begin{eqnarray*}
\langle y(T_2;\zeta,u^\zeta), \;\Phi\rangle-\langle\zeta, \varphi(T_1;\; T_2, \Phi)\rangle
=\int_{T_1}^{T_2}\langle \mathds{1}_\omega u^\zeta, \varphi(t;\; T_2, \Phi) \rangle dt.
\end{eqnarray*}
This, together with  (\ref{e202}), yields that for any $\zeta\in L^2(\Omega)$,
\begin{eqnarray*}
&&\langle\zeta, \varphi(T_1; T_2, \Phi)\rangle\nonumber\\
&\leq& \frac{\|P_My(T_2;\zeta,u^\zeta)\|^2}{2\varepsilon^2}
+\frac{(T_2-T_1)\|u\|_\omega^2}{2[C(\varepsilon, M, T_2-T_1)]^2}\nonumber\\
&&+\frac{\varepsilon^2}{2}\|\Phi\|^2+\frac{[C(\varepsilon, M, T_2-T_1)]^2}{2(T_2-T_1)}\Big\|\chi_\omega\int_{T_1}^{T_2} \varphi(t; T_2, \Phi) dt\Big\|^2\nonumber\\
&\leq& \frac{1}{2}\|\zeta\|^2+\frac{\varepsilon^2}{2}\|\Phi\|^2
+\frac{[C(\varepsilon, M, T_2-T_1)]^2}{2(T_2-T_1)}\Big\|\chi_\omega\int_{T_1}^{T_2} \varphi(t; T_2, \Phi) dt\Big\|^2.
\end{eqnarray*}
By taking $\zeta=\varphi(T_1; T_2, \Phi)$ in the above, we
are led to (\ref{e213}).

\vskip 5pt

\noindent\emph{Step 3. We prove that  the constant $C(\varepsilon, M, T_2-T_1)$
in (\ref{e213}) and the  constant $C(\varepsilon, M, T_2-T_1)$ in (\ref{e202}) are the same. }

Indeed, this can be seen from Step 1 and Step 2 easily.

Hence, we end the proof of Lemma \ref{lem203}.
$\hfill\square$

\section{Minimal norm controls }
\setcounter{equation}{0}
 In this section, we will study properties of time-invariant controls which have the minimal norm among all time-invariant controls
realizing the partially null approximate controllability.
We also study some minimal norm impulse control problems.
These will help us to build up  the feedback law.
Throughout this section, we arbitrarily fix $T_1$ and $T_2$  so that  $0\leq T_1<T_2$.

\subsection{Basic properties on minimal norm controls }
Arbitrarily fix  $\zeta\in L^2(\Omega)$, $M\in \mathds{N}^+$ and $\varepsilon>0$.
Consider the minimal norm  control problem:
\begin{equation}\ \ \ \ \ \ \ \label{e301}
(SNP) \;\;\;\;\ \ \ \mathcal{N}:=\inf\{\sqrt{T_2-T_1}\|f\|_\omega\; :\; f\in L^2(\omega) \ \mbox{s.t.}\ \|P_My(T_2; \zeta, f)\|\leq \varepsilon\|\zeta\|\},
\end{equation}
where $y(\cdot; \zeta, f)$ is the solution to the equation:
\begin{equation}\label{e302}
 \left\{\begin{array}{ll}
 y'(t)- A y(t)=\mathds{1}_\omega f,
 \   \ \ t\in (T_1 ,T_2),\\
  y(T_1)=\zeta.
\end{array}\right.
\end{equation}
We simply write  $y(\cdot; f)$ for the solution to (\ref{e302}), when there is no risk causing any confusion.
The studies of $(SNP)$ are related to the following optimization problem:
\begin{equation}\label{gswang3.3}
(P)\;\;\;\;\;\;\;\;\;\;\;\ \  \inf_{\Phi\in X_M}G(\Phi),
\end{equation}
where the functional $G: X_M\rightarrow \mathds{R}$ is defined by
\begin{equation}\label{gswang3.4}
\ \ \ \ \ G(\Phi)=\frac{1}{2(T_2-T_1)}\Big\|\int_{T_1}^{T_2}\mathds{1}_\omega^*e^{A(T_2-t)}\Phi dt\Big\|_\omega^2+\langle \zeta, e^{A(T_2-T_1)}\Phi\rangle+\varepsilon \|\zeta\|\|\Phi\|,\;\;
\Phi\in X_M.
\end{equation}
\begin{lemma}\label{lem301}
The following conclusions are true:

(i) The functional $G$ is coercive and strictly convex. It has a unique minimizer $\Phi^\zeta$ over $X_M$;

(ii) Let $\Phi^\zeta$ be the minimizer of $G$. Then $\Phi^\zeta=0$ if and only if $\|P_My(T_2; 0)\|\leq \varepsilon\|\zeta\|$.
\end{lemma}
\begin{proof}
(i) By (\ref{e212}), we can use a standard  method (see, for instance, the proof of Theorem 4.3 in \cite{03}) to prove  that $G$ is coercive and strictly convex.
This implies that $G$  has an unique minimizer (see Theorem 2D in \cite{39}).

(ii) This conclusion  can be proved by a very similar way to that in the proof of Lemma 3.6 (c) in \cite{01}. We omit the details here. Hence, we end the proof of Lemma
\ref{lem301}.
\end{proof}

\begin{theorem}\label{th301}
The following conclusions are true:\\
(i) The problem $(SNP)$ has a unique minimal norm control $f^\zeta$;\\
(ii) The minimal norm control $f^\zeta$ to  $(SNP)$ is zero if and only if
$\|P_My(T_2; 0)\|\leq \varepsilon\|\zeta\|$;\\
(iii) The  minimal norm control to the problem $(SNP)$ can be expressed by
\begin{equation}\label{e303}
f^\zeta=\mathds{1}_\omega^*\frac{1}{T_2-T_1}\int_{T_1}^{T_2}e^{A(T_2-t)}\Phi^\zeta dt,
\end{equation}
where $\Phi^\zeta$ is the unique minimizer of the functional  $G$ given by  (\ref{gswang3.4}).
\end{theorem}
\begin{proof}
(i) Write
\begin{equation}\label{ggwang3.10}
\mathcal{F}_{ad}=\{f\in L^2(\omega)\;:\;\|P_My(T_2; f)\|\leq \varepsilon \|\zeta\|\}.
\end{equation}
 By Theorem \ref{th201}, we see that $\mathcal{F}_{ad}\neq \emptyset$. Moreover, one can easily check that $\mathcal{F}_{ad}$ is weakly closed in $L^2(\omega)$. Then by a standard argument, we can show that problem $(SNP)$ has a minimal norm control. Moreover, by Parallelogram Law, one can prove that it is unique.

(ii) By (i) of this theorem, $(SNP)$ has an unique minimal norm  control. Then by  (\ref{e301}), we can easily verify the conclusion (ii) of this theorem.

(iii) We will show that  $f^\zeta$,  given by (\ref{e303}),  is the minimal norm control to $(SNP)$.

In  the case when  $\|P_My(T_2; 0)\|\leq \varepsilon \|\zeta\|$, it follows by (ii) of Lemma \ref{lem301} that $\Phi^\zeta=0$. Moreover, by  (ii) and (i)  of this theorem, we find that
the unique minimal norm control of $(SNP)$ is zero.
Clearly,  $f^\zeta$, given by   (\ref{e303}), is zero  in this case.
 Hence, it is the minimal norm control to $(SNP)$.

In the case that
$\|P_My(T_2; 0)\|> \varepsilon \|\zeta\|$,
 it suffices to show two facts: First,  $f^\zeta\in \mathcal{F}_{ad}$.
Second,  $\|f^\zeta\|_\omega\leq \|f\|_\omega$ for all $f\in \mathcal{F}_{ad}$.

To prove the first fact, we use  (ii) of Lemma \ref{lem301} to see that
 $\Phi^\zeta\neq 0$. Thus the Euler-Lagrange equation of $G$ associated with
 $\Phi^\zeta$ is as:
\begin{eqnarray}\label{e304}
&&\Big\langle\mathds{1}_\omega^*\frac{1}{\sqrt{T_2-T_1}}
\int_{T_1}^{T_2}e^{A(T_2-t)}\Phi^\zeta dt, \; \mathds{1}_\omega^*\frac{1}{\sqrt{T_2-T_1}}\int_{T_1}^{T_2}e^{A(T_2-t)}\Phi dt\Big\rangle_\omega\nonumber\\
&&+\langle \zeta, e^{A(T_2-T_1)}\Phi\rangle+\varepsilon \|\zeta\|\langle\frac{\Phi^\zeta}{\|\Phi^\zeta\|}, \;\Phi\rangle=0\;\;\mbox{for all}\;\; \Phi\in X_M.
\end{eqnarray}
From  (\ref{e303}) and (\ref{e304}), we can easily verify  that
\begin{eqnarray*}
P_My(T_2; f^\zeta)=P_M\Big(e^{A(T_2-T_1)}\zeta
+\int_{T_1}^{T_2}e^{A(T_2-t)}\mathds{1}_\omega f^\zeta dt\Big)=-\varepsilon \|\zeta\|\frac{\Phi^\zeta}{\|\Phi^\zeta\|}.
\end{eqnarray*}
This implies that
$\|P_My(T_2;f^\zeta)\|\leq\varepsilon \|\zeta\|$.
Hence, we have that $f^\zeta\in\mathcal{F}_{ad}$.

 To show the second fact,  we arbitrarily fix $f\in \mathcal{F}_{ad}$. Then it follows by (\ref{ggwang3.10}) that
$\|P_My(T_2; f)\|\leq \varepsilon \|\zeta\|$.
This, along with  (\ref{e303}) and  (\ref{e304}) (where $\Phi=\Phi^\zeta$), yields
 that
\begin{eqnarray}\label{e308}
(T_2-T_1)\|f^\zeta\|_\omega^2
&=&-\varepsilon \|\zeta\|\|\Phi^\zeta\|-\langle e^{A(T_2-T_1)}\zeta, \; \Phi^\zeta\rangle\nonumber\\
&\leq& \langle y(T_2; f), \;\Phi^\zeta\rangle-\langle e^{A(T_2-T_1)}\zeta, \Phi^\zeta\rangle\nonumber\\
&=&\Big\langle \sqrt{T_2-T_1} f, \;\sqrt{T_2-T_1}\mathds{1}_\omega^*\frac{1}{T_2-T_1}
\int_{T_1}^{T_2}e^{A(T_2-t)}\Phi^\zeta dt\Big\rangle_\omega\nonumber\\
&\leq& \frac{1}{2}(T_2-T_1)\|f\|_\omega^2+\frac{1}{2}(T_2-T_1)\|f^\zeta\|_\omega^2,
\end{eqnarray}
which leads to the second fact as desired.

Finally, from the above two facts, we see that $f^\zeta$ is the minimal norm control to $(SNP)$ in the second case.
Hence, we end the proof of Theorem \ref{th301}.
\end{proof}

\begin{remark}
The control $\tilde{u}$ given by  Theorem \ref{th201} is an admissible control to Problem $(SNP)$, i.e., $\tilde{u}\in \mathcal{F}_{ad}$. Then  it follows by (\ref{e208}) that
$\|f^\zeta\|_{L^2(T_1, T_2; L^2(\omega))}\leq C(\varepsilon, M, T_2-T_1)\|\zeta\|$,
where $f^\zeta$ is the minimal norm control to $(SNP)$ and
$C(\varepsilon, M, T_2-T_1)$ is the constant given by (\ref{e209}).
 Thus, we obtain an upper bound for the minimal norm control to $(SNP)$. However,
 this bound is not good enough to build up the desired feedback law. So we aim to
 seek a better bound for  $f^\zeta$  in the next subsection.
 \end{remark}
\subsection{Estimates for  minimal norm controls }
In this subsection, we shall give some  estimates for  the minimal norm control $f^\zeta$ to the problem $(SNP)$.  These estimates will be used
to build up the feedback law and to get lower and upper bounds for the norm of the feedback law.
Before giving the main theorem of this subsection, we show  the following lemma which gives some properties on the matrix defined by (\ref{wangggs2.19}).
Similar results for the heat equation without any potential was given in \cite{02}.
\begin{lemma}\label{lm302}
For any $M\in \mathds{N}^+$, the matrix $B_M$, given by (\ref{wangggs2.19}) with $k=M$, is positive definite. Furthermore, if $M\in \mathbb{N}^+$ satisfies that  $\lambda_M\geq 0$,
then
\begin{equation}\label{e312}
|\alpha|_M^2\leq\alpha' B_M^{-1}\alpha\leq e^{C_0(1+\gamma_0^{\frac{2}{3}}+\sqrt{\lambda_{M}})}|\alpha|_M^2\;\;\mbox{for any}\;\;
\alpha\in \mathbb{R}^M.
\end{equation}
Here $\gamma_0$ is given by (\ref{wang1.2}) and $C_0=C(\Omega,\omega)$.
 \end{lemma}
 \begin{proof}
 Let $M\in \mathbb{N}^+$. Since $\mathds{1}^*_\omega\xi_1, \dots,\mathds{1}^*_\omega\xi_M$ are linear independent in $L^2(\omega)$
 (see, for instance, Page 38 in \cite{22}),
we can easily verify that  the symmetric matrix $B_M$ is  positive definite.

Now we arbitrarily fix $M\in \mathbb{N}^+$, with $\lambda_M\geq 0$, and $\alpha=(a_1,a_2,\cdots a_M)'\in \mathbb{R}^M$.
To  show the first inequality in (\ref{e312}), we write  $\sigma(B_M)$ for  the set of all eigenvalues of $B_M$. Let
$\underline{\lambda}:=\min\sigma(B_M)$ and $\bar{\lambda}:=\max\sigma(B_M)$.
Then we have that
\begin{eqnarray*}
\alpha' B_M\alpha
=\Big\|\sum_{i=1}^Ma_i\mathds{1}_\omega^*\xi_i\Big\|_\omega^2
=\Big\|\mathds{1}_\omega^*\sum_{i=1}^Ma_i\xi_i\Big\|_\omega^2\leq \Big\|\sum_{i=1}^Ma_i\xi_i\Big\|^2=|\alpha|_M^2.
\end{eqnarray*}
This implies that $\bar{\lambda}\leq 1$. So
$\min\sigma(B_M^{-1})\geq 1$, which leads to the first inequality in (\ref{e312}).

We now show the second inequality in (\ref{e312}). By the same way as that used in the proof of (ii) of Theorem 2.1 in \cite{01}, we can find a constant $C_0=C(\Omega,\omega)$ so that
$$
|\alpha|_M^2\leq e^{C_0\big(1+\gamma_0^{\frac{2}{3}}+\sqrt{\lambda_{M}}\big)}
\int_\omega|\sum_{i=1}^Ma_i\xi_i|^2dx
=e^{C_0\big(1+\gamma_0^{\frac{2}{3}}+\sqrt{\lambda_{M}}\big)}\alpha' B_M\alpha.
 $$
 This implies that $\underline{\lambda}\geq e^{-C_0(1+\gamma_0^{\frac{2}{3}}+\sqrt{\lambda_{M}})}$. Thus,
$\max\sigma(B_M^{-1})\leq e^{C_0\big(1+\gamma_0^{\frac{2}{3}}+\sqrt{\lambda_{M}}\big)}$,  which
leads to the second inequality in (\ref{e312}).
This
 ends the proof of Lemma \ref{lm302}.
 \end{proof}

The main theorem of this subsection is as:
\begin{theorem}\label{th302}
Let $\zeta\in L^2(\Omega)$ and $M\in \mathbb{N}^+$ satisfy that  $\|P_Me^{A(T_2-T_1)}\zeta\|>\varepsilon \|\zeta\|$ and $\lambda_{M}\geq0$. Then the minimal norm control $f^\zeta$ to  the problem $(SNP)$ satisfies the estimates:
\begin{equation}\label{e313}
 \begin{array}{ll}\|f^\zeta\|_{\omega}\geq
 \frac{1}{\alpha_1}(\|P_Me^{A(T_2-T_1)}\zeta\|-\varepsilon\|\zeta\|);\\
  \|f^\zeta\|_{\omega}\leq e^{\tilde{C}_1\big(1+\gamma_0^{\frac{2}{3}}+\sqrt{\lambda_{M}}\big)}
  (\frac{\gamma_0}{1-e^{-\gamma_0 (T_2-T_1)}}+\frac{\varepsilon}{\alpha_M})\|\zeta\|,
\end{array}
\end{equation}
where $\tilde{C}_1=C_0/2$, with $C_0$ given by Lemma \ref{lm302}, and
\begin{equation}\label{ggsswang3.17}
\alpha_j=\int_{T_1}^{T_2}e^{-\lambda_j(T_2-t)}dt,\; j=1, 2, \cdots, M.
 \end{equation}
\end{theorem}
\begin{proof}
Let $\Phi^\zeta$ be the minimizer of $G$ (given by (\ref{gswang3.4})).
Since $\|P_Me^{A(T_2-T_1)}\zeta\|>\varepsilon \|\zeta\|$, it follows from Lemma \ref{lem301} that $\Phi^\zeta\neq 0$.
Write
$\zeta:=\sum_{i=1}^\infty a_i\xi_i$ and $\Phi^\zeta:=\sum_{i=1}^M b_i\xi_i$.
 Let
$$
\tilde{\beta}:=\left(
                        \begin{array}{c}
                          e^{-\lambda_1(T_2-T_1)}a_1 \\
                          \vdots\\

                          e^{-\lambda_M(T_2-T_1)}a_M\\
                        \end{array}
                      \right)+\frac{\varepsilon\|\zeta\|}{\|\Phi^\zeta\|}\left(
       \begin{array}{c}
         b_1 \\
         \vdots \\
         b_M \\
       \end{array}
     \right)\;\;\mbox{and}\;\;D:=\mathrm{diag}(\alpha_i)_{i=1}^M.
     $$
Then by  (\ref{e304}), we see that
\begin{equation}\label{e314}
DB_MD(b_1,\dots,b_M)'=-(T_2-T_1)\tilde{\beta}.
\end{equation}
Moreover, one can easily check that
\begin{eqnarray}\label{wang2018520}
 (\ref{e304})  \Leftrightarrow (\ref{e314}),
 \;\;\mbox{when}\;\;\Phi^\zeta\neq 0.
\end{eqnarray}
(We would like to mention: (\ref{wang2018520}) will not be used in this proof, but will be used in Section 4.)
From (\ref{e314}), it follows that
\begin{equation}\label{e315}
(b_1,\dots,b_M)'=-(T_2-T_1)D^{-1}B_M^{-1}D^{-1}\tilde{\beta}.
\end{equation}
By (iii) of Theorem \ref{th301} and (\ref{e315}), after some  simple calculations, we
find that
\begin{eqnarray}\label{e316}
\|f^\zeta\|^2_{\omega}
=\frac{1}{(T_2-T_1)^2}(b_1, \cdots, b_M)DB_MD(b_1,\dots,b_M)'
=\tilde{\beta}^TD^{-1}B_M^{-1}D^{-1}\tilde{\beta}.
\end{eqnarray}
From (\ref{e316}) and  Lemma \ref{lm302}, we see that
\begin{eqnarray}\label{e317}
\|f^\zeta\|^2_{\omega}\geq |D^{-1}\tilde{\beta}|^2_M\geq\frac{1}{\alpha_1^2}(\|P_Me^{A(T_2-T_1)}\zeta\|-\varepsilon\|\zeta\|)^2.
\end{eqnarray}
This gives the first inequality of (\ref{e313}).

To show the second inequality of (\ref{e313}), we first  notice that $-\lambda_1\leq\gamma_0$. Then by (\ref{e316}) and Lemma \ref{lm302}, we obtain that
\begin{eqnarray}\label{e318}
\|f^\zeta\|^2_{\omega}
&\leq&e^{C_0\big(1+\gamma_0^{\frac{2}{3}}
+\sqrt{\lambda_{M}}\big)}|D^{-1}\tilde{\beta}|_M^2\nonumber\\
&\leq&e^{C_0\big(1+\gamma_0^{\frac{2}{3}}
+\sqrt{\lambda_{M}}\big)}\Big(\frac{e^{-\lambda_1(T_2-T_1)}}{\alpha_1}
+\frac{\varepsilon}{\alpha_M}\Big)^2\|\zeta\|^2\nonumber\\
&\leq&e^{C_0\big(1+\gamma_0^{\frac{2}{3}}
+\sqrt{\lambda_{M}}\big)}\Big(\frac{\gamma_0}{1-e^{-\gamma_0 (T_2-T_1)}}+\frac{\varepsilon}{\alpha_M}\Big)^2\|\zeta\|^2,
\end{eqnarray}
which leads to  the second inequality of (\ref{e313}). This  completes the proof
of Theorem \ref{th302}.
\end{proof}

\subsection{Minimal norm impulse control problems}
In this subsection, we study some minimal norm impulse control problems. We will omit the detailed proofs, since they are very similar to those in two subsections above.
Let $\omega_1\subset \Omega$ be the open subset given  in subsection 1.2 and
let $\tau\in (T_1, T_2)$. Given $\zeta\in L^2(\Omega)$ and $h\in L^2(\omega_1)$, write $z(\cdot; \zeta, h)$ for the solution of
the impulse controlled equation:
\begin{equation}\label{e320}
 \left\{\begin{array}{ll}
 z'(t)- A z(t)=0,
 \   \ \ t\in (T_1 ,T_2)\backslash\{\tau\},\\
 z(T_1)=\zeta,\\
  z(\tau)=z(\tau-)+\mathds{1}_{\omega_1} h.\end{array}\right.
\end{equation}
\begin{definition}\label{def202}
 The equation (\ref{e320}) is said to have the partially null approximate controllability with a cost, if for any  $M\in \mathds{N}^+$ and $\varepsilon >0$,
 there is a constant $\tilde{C}(\varepsilon, M, T_2-T_1)$ so that  for any $\zeta\in L^2(\Omega)$,
 there is a control $h^{\zeta}\in L^2(\omega_1)$ satisfying that
$$
\frac{1}{[\tilde{C}(\varepsilon, M, T_2-T_1)]^2}\|h^{\zeta}\|^2_{\omega_1}
+\frac{1}{\varepsilon^2}\|P_Mz(T_2;\zeta,h^{\zeta})\|^2\leq \|\zeta\|^2.
$$
\end{definition}
Next,  we arbitrarily fix   $\zeta\in L^2(\Omega)$, $\varepsilon>0$ and $M\in \mathds{N}^+$.
Consider
the minimal norm impulse control problem:
\begin{equation}\label{e319}
(INP) \ \ \  \widetilde{\mathcal{N}}:=\inf\{\|h\|_{\omega_1}\; :\; h\in L^2(\omega_1)\ \mathrm{s.t.}\ \|P_Mz(T_2;\zeta, h)\|\leq \varepsilon\|\zeta\|\}.
\end{equation}
It was proven that the equation (\ref{e320}) has the null approximate controllability with a cost (see Theorem 3.1 in \cite{01}). Consequently, it has the  partially null approximate controllability with a cost.
By the same arguments to those used to prove properties of the problem $(SNP)$ in the last subsection, we can get the  results in the next theorem. (These results are quite  similar to those in Theorem 3.4 in \cite{01}.)
\begin{theorem}\label{th303}
The following conclusions are true:\\
(i) The problem $(INP)$ has a unique minimal norm control $h^\zeta$;\\
(ii) When $h^\zeta$ is the minimal norm control to $(INP)$,
$h^\zeta=0$ iff $\|P_M z(T_2;\zeta, 0)\|\leq \varepsilon\|\zeta\|$;\\
(iii) The minimal norm control to $(INP)$ can be expressed  by
\begin{equation}\label{e321}
h^\zeta=\mathds{1}_{\omega_1}^*e^{A(T_2-\tau)}\Psi^\zeta,
\end{equation}
where $\Psi^\zeta$ is the unique minimizer of
the functional:
$$
H(\Psi)=\frac{1}{2}\|\mathds{1}_{\omega_1}^*e^{A(T_2-\tau)}\Psi \|_{\omega_1}^2+\langle \zeta, e^{A(T_2-T_1)}\Psi\rangle+\varepsilon \|\zeta\|\|\Psi\|,\;\;\Psi\in X_M.
$$
\end{theorem}
By the very similar methods to those used in the proof of Theorem \ref{th302}, we can obtain what follows:
\begin{theorem}\label{th304}
 Let $\zeta\in L^2(\Omega)$ and $M\in \mathbb{N}^+$
satisfy that $\|P_Me^{A(T_2-T_1)}\zeta\|>\varepsilon \|\zeta\|$ and $\lambda_{M}\geq 0$. Then the minimal norm control $h^\zeta$ to $(INP)$  satisfies the estimates:
\begin{equation}\label{e322}
 \begin{array}{ll}\|h^\zeta\|_{\omega_1}\geq e^{\lambda_1(T_2-\tau)}(\|P_Me^{A(T_2-T_1)}\zeta\|-\varepsilon\|\zeta\|),\\
  \|h^\zeta\|_{\omega_1}\leq e^{\tilde{C}_1\big(1+\gamma_0^{\frac{2}{3}}+\sqrt{\lambda_{M}}\big)}(e^{-\lambda_1 (\tau-T_1)}+e^{\lambda_M (T_2-\tau)}\varepsilon)\|\zeta\|.
\end{array}
\end{equation}
\end{theorem}

\section{Sampled-data output feedback  stabilization}
\setcounter{equation}{0}
This section presents and proves our main theorems of this paper.

\subsection{Design of the feedback  law}
Recall that $\gamma_0$ is given by (\ref{wang1.2}),  $\tilde{C}_1$ is given  in Theorem \ref{th302}
and $m$, as well as $\lambda_{m+1}$, is given by (\ref{wanggs1.3}). Let
\begin{equation}\label{liu4.1}
\hat{c}_p=\max\{0, \lambda_{m+1}-3\gamma_0\}.
\end{equation}
Given $\gamma >0$ and $T>0$,   we let
\begin{equation}\label{wang4.1}
N:=\max\left\{j\in \mathds{N}^+\; :\;  \lambda_j< 2\gamma+{(\ln 9)}/{T}\right\};
\end{equation}
\begin{equation}\label{wang4.2}
M:=\max\{j\in \mathds{N}^+\; :\; \lambda_j< C(\gamma,T)\}
\end{equation}
with
\begin{equation}\label{WANG4.3}
\;\;\;\;\;\;C(\gamma, T):=\left(\frac{\tilde{C}_1+\sqrt{\tilde{C}_1^2
+2T\big[\mathrm{ln}(9\sqrt{N})+\tilde{C}_1(1+\gamma_0^{\frac{2}{3}})\big]
+(4\gamma+3\gamma_0)T^2}}{T}\right)^2+\hat{c}_p;
\end{equation}
\begin{equation}\label{wang4.3}
\varepsilon_0:=\frac{1}{9\sqrt{N}}e^{-(2\gamma+\frac{3}{2}\gamma_0+\hat{c}_p) T}.
\end{equation}
 By (\ref{liu4.1}) and (\ref{WANG4.3}), one can easily check that
 $\lambda_{m+1}< C(\gamma, T)$ and $2\gamma+{(\ln 9)}/{T}\leq C(\gamma, T)$.
 From these,  (\ref{wang4.1}) and (\ref{wang4.2}), we have that
 \begin{equation}\label{wang4.4}
 \lambda_M>0,\;\;M\geq N\;\mbox{and}\;\;M\geq m+1.
 \end{equation}
Let  $f_j\in L^2(\omega)$, with  $j=1, 2, \cdots, N$, be
 the  optimal controls to the problem
$(SNP)$ given by (\ref{e301}), with
\begin{equation}\label{YUANGYUAN4.6}
T_1=0,\; T_2=T/2, \;\zeta=\xi_j,\; M\;\mbox{given by}\;(\ref{wang4.2}), \varepsilon=\varepsilon_0\;\mbox{given by}\;(\ref{wang4.3}).
\end{equation}
   Let
 $h_j\in L^2(\omega_1)$, with  $j=1, 2, \cdots, N$, be the  optimal control to the problem  $(INP)$
 given by  (\ref{e319}), where
 \begin{equation}\label{YUANGYUAN4.7}
 T_1=0,\;T_2=T/2,\;\tau=T/4,\;\zeta=\xi_j,\;M\;\mbox{given by}\;(\ref{wang4.2}), \varepsilon=\varepsilon_0\;\mbox{given by}\;(\ref{wang4.3}).
 \end{equation}
  We now define the feedback law $\mathcal{F}_T: L^2(\omega_1)\rightarrow L^2(\omega)$ as follows:
\begin{equation}\label{e402}
\mathcal{F}_T(v)=-\sum_{j=1}^{N}\langle \mathds{1}_{\omega_1}v, \mathds{1}_{\omega_1}h_{j}\rangle f_{j}, \;\; v\in L^2(\omega_1).
\end{equation}
\begin{remark}\label{wangremark4.1}
(i) This feedback law depends on both $T$ and $\gamma$. However,  we only concern its dependence on $T$ in the current study. Thus, we denote it by $\mathcal{F}_T$.

(ii) We simply write $\|\mathcal{F}_T\|$ for $\|\mathcal{F}_T\|_{\mathcal{L}(L^2(\omega_1), L^2(\omega))}$.

(iii) The vectors $f_j$ and $h_j$ are all non-zero. Indeed, one can easily check that
\begin{equation}\label{wangpuchong4.5}
\ \ \ \ \ \ \ \ \   \|P_Me^{A\frac{T}{2}}\xi_j\|=e^{-\frac{\lambda_jT}{2}}>e^{-(2\gamma+\frac{\ln 9}{T})\frac{T}{2}}=\frac{1}{3}e^{-\gamma T}>\varepsilon_0\|\xi_j\|\;\mbox{for all}\;j=1, 2, \dots, N.
\end{equation}
Then we can apply (ii) of Theorem \ref{th301} and (ii) of Theorem \ref{th303} to see that $f_j\neq 0$ and $h_j\neq 0$ for all $j=1, 2, \dots, N$.
\end{remark}
With the aid of the above $\mathcal{F}_T$, we can define the following  closed-loop equation:
\begin{eqnarray}\label{wang4.6}
 \ \ \ \ \ \ \  \ \ \left\{\begin{array}{ll}
 y'(t)- Ay(t)=\mathds{1}_\omega \sum_{i=0}^{\infty}\chi_{[(2i+1)T, (2i+\frac{3}{2})T)}(t)\mathcal{F}_T(\mathds{1}^*_{\omega_1}y((2i+\frac{3}{4})T)),  \ t>0,\\
  y(0)\in L^2(\Omega).
\end{array}\right.
\end{eqnarray}
(It is exactly the same as  (\ref{e102}).)
The first main theorem of this paper is as:

\begin{theorem}\label{th401}
 Let  $\mathcal{F}_T$ be given  by (\ref{e402}), with $T>0$ and $\gamma>0$. Then
any solution $y(\cdot)$ to the equation (\ref{wang4.6}) satisfies that
\begin{equation}\label{e405}
\|y(t)\|\leq \Big(1+\frac{T}{2}\|\mathcal{F}_T\|\Big)e^{(2\gamma_0+3\gamma)T}e^{-\gamma t}\|y(0)\|\;\;\mbox{for all}\;\; t>0.
\end{equation}
\end{theorem}
\begin{proof} Arbitrarily fix $T>0$ and $\gamma>0$. Let $y(\cdot)$ be a solution to
(\ref{wang4.6}).
 We first show that
 \begin{equation}\label{e404}
\|y((2n+2)T)\|\leq e^{-2\gamma T}\|y(2nT)\|\;\;\mbox{for all}\;\; n\in \mathds{N}.
\end{equation}
  Define the operator $\mathcal{R}: L^2(\omega_1)\rightarrow X_N$  by
 $\mathcal{R}(v):=-\sum_{j=1}^{N}
 \langle\mathds{1}_{\omega_1}v,\mathds{1}_{\omega_1}h_{j}\rangle \xi_j$ for each $v\in L^2(\omega_1)$
 and the operator $\mathcal{Q}: X_N\rightarrow L^2(\omega)$ by
 $\mathcal{Q}(\zeta)=\sum_{j=1}^{N}\langle \zeta, \xi_j\rangle f_j$
 for each $\zeta\in X_N$.
 Then by (\ref{e402}), we see that $\mathcal{F}_T=\mathcal{Q}\circ \mathcal{R}$.
 From this, using the fact:
\begin{eqnarray*}
y((2n+2)T)&=&e^{TA}y((2n+1)T)\\ \nonumber
&+&e^{\frac{AT}{2}}
\int_{(2n+1)T}^{(2n+\frac{3}{2})T}
e^{((2n+\frac{3}{2})T-t)A}
\mathds{1}_{\omega}\mathcal{F}_T(\mathds{1}_{\omega_1}^*y((2n+{3}/{4})T))dt,
\end{eqnarray*}
we obtain that
\begin{eqnarray}\label{e408}
y((2n+2)T)=I_1+I_2+I_3,
\end{eqnarray}
where
\begin{eqnarray*}
I_1:=e^{TA}[y((2n+1)T)-P_N y((2n+1)T)],
\end{eqnarray*}
\begin{eqnarray*}
I_2:=e^{TA}[P_N y((2n+1)T)-\mathcal{R}(\mathds{1}_{\omega_1}^*y((2n+{3}/{4})T))],
\end{eqnarray*}
\begin{eqnarray*}
I_3&:=&e^{TA}\mathcal{R}(\mathds{1}_{\omega_1}^*y((2n+{3}/{4})T))
\\ \nonumber
&+&e^{\frac{T}{2}A}
\int_{(2n+1)T}^{(2n+{3}/{2})T}
e^{((2n+\frac{3}{2})T-t)A}\mathds{1}_\omega\mathcal{Q}\circ \mathcal{R}(\mathds{1}_{\omega_1}^*y((2n+{3}/{4})T))dt.
\end{eqnarray*}
We will prove (\ref{e404})
 by  the following three steps:

\noindent{\it Step 1. We estimate  $I_1$.}

Because
$y(t)=e^{A(t-2nT)}y(2nT)$ for each $t\in [2nT, (2n+1)T]$,
 we can use (\ref{wang4.1}) to see that
\begin{eqnarray}\label{e409}
\ \ \ \ \ \|I_1\|
=\|e^{2TA}(I-P_N)y(2nT)\|
\leq e^{-2\lambda_{N+1}T}\|y(2nT)\|
\leq \frac{1}{9}e^{-2\gamma T}\|y(2nT)\|.
\end{eqnarray}

\noindent{\it Step 2. We estimate $I_2$.}

Arbitrarily fix $1\leq j\leq N$. Let $\psi_{j}(\cdot)$ be the solution of the equation
(\ref{e320}) where $T_1=0$, $T_2=T/2$, $\tau=T/4$, $\zeta=\xi_j$ and  $h=h_j$.
Several facts are given in order.

Fact one: Since $h_j$ is optimal to the problem $(INP)$
given by  (\ref{e319}) with (\ref{YUANGYUAN4.7}), it is admissible to this problem. So we have that
\begin{equation}\label{e410}
\|P_M\psi_{j}({T}/{2})\|\leq \varepsilon_0\|\xi_j\|.
\end{equation}

Fact two: It is clear  that
\begin{equation}\label{ggsswang4.14}
\psi_{j}({T}/{2})=e^{A\frac{T}{2}}\xi_j+e^{A\frac{T}{4}}\mathds{1}_{\omega_1}h_{j}.
\end{equation}

Fact three: Since $N\leq M$, we have that
\begin{equation}\label{wangbucong4.15}
(I-P_M)e^{A\frac{T}{2}}\xi_j=0\;\;\mbox{for all}\;\;1\leq j\leq N.
\end{equation}

Fact four: By (\ref{wangpuchong4.5}),  we apply Theorem \ref{th304},
with (\ref{YUANGYUAN4.7}),
to obtain that
\begin{equation}\label{wangbucong4.16}
\|h_j\|_{\omega_1}\leq  e^{\tilde{C}_1(1+\gamma_0^{\frac{2}{3}}+\sqrt{\lambda_{M}})}(e^{-\lambda_1T/4}+e^{\lambda_M T/4}\varepsilon_0)\|\xi_j\|\;\;\mbox{for all}\;\;j=1,\dots, N.
\end{equation}

Now, from (\ref{ggsswang4.14}), (\ref{wangbucong4.15}) and (\ref{wangbucong4.16}), we find that
\begin{eqnarray}\label{e411}
\|(I-P_M)\psi_{j}({T}/{2})\|
&\leq&e^{\frac{-\lambda_{M+1}T}{4}}\|h_j\|_{\omega_1}\nonumber\\
&\leq&e^{\frac{-\lambda_{M+1}T}{4}}
e^{\tilde{C}_1\big(1+\gamma_0^{\frac{2}{3}}
+\sqrt{\lambda_{M}}\big)}\big(e^{\frac{\gamma_0T}{4}}
+e^{\frac{\lambda_{M}T}{4}}\varepsilon_0\big)\nonumber\\
&\leq&e^{\tilde{C}_1\big(1+\gamma_0^{\frac{2}{3}}
+\sqrt{\lambda_{M}}\big)}\big(e^{\frac{\gamma_0T}{4}}+\varepsilon_0\big).
\end{eqnarray}
Then by  (\ref{e410}), (\ref{e411}) and  the definitions of $\mathcal{R}$ and $\psi_{j}$, we see that
\begin{eqnarray}\label{e412}
\ \ \ &&\|P_N y((2n+1)T)-\mathcal{R}(\mathds{1}_{\omega_1}^*y((2n+{3}/{4})T))\|\nonumber\\
&\leq&\Big\|\sum_{j=1}^{N}\big\langle y(2nT), \;e^{A\frac{T}{2}}P_M\psi_{j}({T}/{2}) \big\rangle\xi_j\Big\|\nonumber\\
 &+&\Big\|\sum_{j=1}^{N}\big\langle y(2nT), \; e^{A\frac{T}{2}}(I-P_M)\psi_{j}({T}/{2}) \big\rangle\xi_j\Big\|\nonumber\\
&\leq&\sqrt{N}\Big[e^{\frac{\gamma_0 T}{2}}\varepsilon_0
+e^{\frac{-\lambda_{M+1}T}{2}}
e^{\tilde{C}_1\big(1+\gamma_0^{\frac{2}{3}}
+\sqrt{\lambda_{M}}\big)}(e^{\frac{\gamma_0T}{4}}+\varepsilon_0)\Big]
\|y(2nT)\|\nonumber\\
&\leq&\frac{1}{3}e^{-(2\gamma+\gamma_0) T}\|y(2nT)\|.
\end{eqnarray}
In the last inequality in (\ref{e412}), we used (\ref{wang4.3}) and the fact that
\begin{equation}\label{add001}
e^{\gamma_0T}\sqrt{N}e^{\frac{
-\lambda_{M+1}T}{2}}e^{\tilde{C}_1
\big(1+\gamma_0^{\frac{2}{3}}+\sqrt{\lambda_{M+1}}\big)}
e^{\frac{\gamma_0 T}{2}}\leq \frac{1}{9}e^{-2\gamma T},
\end{equation}
which follows from  (\ref{wang4.2}) and (\ref{WANG4.3}).
Finally, by  (\ref{e412}), we find that
\begin{eqnarray}\label{e413}
\ \ \ \ \ \ \ \ \ \ \|I_2\|=\|e^{TA}[P_N y((2n+1)T)-\mathcal{R}(\mathds{1}_{\omega_1}^*y((2n+{3}/{4})T))]\|
\leq \frac{1}{3}e^{-2\gamma T}\|y(2nT)\|.
\end{eqnarray}

\noindent{\it Step 3. We estimate  $I_3$.}

Arbitrarily fix $1\leq j\leq N$. Let  $\varphi_{j}$
 be the solution of  (\ref{e302}),
 where   $T_1=0$, $T_2=T/2$, $\zeta=\xi_j$ and $f=f_j$.
 Several facts are given in order.

 Fact one: It is clear that
 \begin{equation}\label{sgsgwang4.21}
 \varphi_{j}({T}/{2})=e^{\frac{T}{2}A}\xi_j
 +\int_0^{\frac{T}{2}}e^{A(\frac{T}{2}-t)}\mathds{1}_{\omega}f_{j}dt.
 \end{equation}

Fact two: Since $f_j\in L^2(\omega)$ is
 the  optimal control to the problem
$(SNP)$ given by (\ref{e301}) with (\ref{YUANGYUAN4.6}),
 we see that $f_j\in L^2(\omega)$ is
an admissible control to this problem. This, along with (\ref{wang4.3}), yields that
\begin{equation}\label{e414}
\|e^{\frac{T}{2}A}P_M\varphi_{j}({T}/{2})\|\leq e^{\frac{\gamma_0T}{2}}\varepsilon_0\|\xi_j\|\leq \frac{1}{9\sqrt{N}}e^{-(2\gamma+\gamma_0) T}.
\end{equation}

Fact three: We have that
\begin{eqnarray}\label{e415}
\|e^{\frac{T}{2}A}(I-P_M)\varphi_{j}({T}/{2})\|\leq  \frac{2}{9\sqrt{N}}e^{-(2\gamma+\gamma_0) T}.
\end{eqnarray}
Indeed, since $M\geq N$ (see (\ref{wang4.4})) and $\xi_j$ is an eigenfunction of $A$, we have that
$(I-P_M)e^{\frac{T}{2}A}\xi_j=0$.
Thus, by (\ref{sgsgwang4.21}), we see that
\begin{eqnarray}\label{wangpuchong4.24}
\|e^{\frac{T}{2}A}(I-P_M)\varphi_{j}({T}/{2})\|
&=&\Big\|e^{\frac{T}{2}A}(I-P_M)
\int_0^{\frac{T}{2}}e^{A(\frac{T}{2}-t)}\mathds{1}_{\omega}f_{j}dt\Big\|\nonumber\\
&\leq& e^{\frac{-\lambda_{M+1}T}{2}}
\int_0^{\frac{T}{2}}e^{-\lambda_{M+1}(\frac{T}{2}-t)}dt\|f_{j}\|_\omega.
\end{eqnarray}
Meanwhile, because of (\ref{wangpuchong4.5}), we can apply   Theorem \ref{th302},
with (\ref{YUANGYUAN4.6}),
 to obtain that
\begin{eqnarray}\label{wangpuchong4.25}
 e^{\frac{-\lambda_{M+1}T}{2}}
\int_0^{\frac{T}{2}}e^{-\lambda_{M+1}(\frac{T}{2}-t)}dt\|f_{j}\|_\omega
&\leq& e^{\frac{-\lambda_{M+1}T}{2}}e^{\tilde{C}_1
\big(1+\gamma_0^{\frac{2}{3}}+\sqrt{\lambda_{M}}\big)}
\Big(\frac{\gamma_0T}{2}\frac{e^{\frac{\gamma_0 T}{2}}}{e^{\frac{\gamma_0 T}{2}}-1}+\varepsilon_0\Big)\nonumber\\
&\leq& e^{\frac{-\lambda_{M+1}T}{2}}e^{\tilde{C}_1
\big(1+\gamma_0^{\frac{2}{3}}+\sqrt{\lambda_{M}}\big)}\Big(e^{\frac{\gamma_0 T}{2}}+\varepsilon_0\Big)\nonumber\\
&\leq&  \frac{2}{9\sqrt{N}}e^{-(2\gamma+\gamma_0) T}.
\end{eqnarray}
(To get  the second inequality in (\ref{wangpuchong4.25}), we used the inequality:   $x<e^x-1$, when $x>0$; to get the last inequality in (\ref{wangpuchong4.25}), we used (\ref{add001}).) Now (\ref{e415}) follows from (\ref{wangpuchong4.24}) and (\ref{wangpuchong4.25}) at once.

Fact four: Simply  write $\zeta_n:=\mathcal{R}(\mathds{1}_{\omega_1}^*y((2n+\frac{3}{4})T))$. Then  by the definitions of $\mathcal{R}$ and $\mathcal{Q}$, we find that
\begin{eqnarray}\label{wangpuchong4.26}
I_3
&=&e^{TA}\sum_{j=1}^N\langle\zeta_n,\xi_j\rangle\xi_j
+e^{\frac{T}{2}A}\int_{0}^{\frac{T}{2}}
e^{(\frac{T}{2}-t)A}\mathds{1}_{\omega}
\mathcal{Q}\Big(\sum_{j=1}^N\langle\zeta_n,\xi_j\rangle\xi_j\Big)dt\nonumber\\
&=&\sum_{j=1}^N\langle\zeta_n,\xi_j\rangle e^{\frac{T}{2}A}[P_M\varphi_{j}({T}/{2})+(I-P_M)\varphi_{j}({T}/{2})].
\end{eqnarray}
From (\ref{wangpuchong4.26}),   (\ref{e414}), (\ref{e415}) and (\ref{e412}), it follows that
\begin{eqnarray}\label{e416}
\|I_3\|&\leq&
\frac{1}{3\sqrt{N}}e^{-(2\gamma+\gamma_0) T} \sqrt{N}\|\zeta_n\|\nonumber\\
&\leq& \frac{1}{3}e^{-(2\gamma+\gamma_0) T} \|P_N y((2n+1)T)\|
+\frac{1}{3}e^{-(2\gamma+\gamma_0) T}\|P_N y((2n+1)T)-\zeta_n\|\nonumber\\
&\leq& \frac{1}{3}e^{-2\gamma T}\|y(2nT)\|+\frac{1}{3}e^{-(2\gamma+\gamma_0) T} \frac{1}{3}e^{-2\gamma T}\|y(2nT)\|\nonumber\\
&\leq& \frac{4}{9} e^{-2\gamma T}\|y(2nT)\|.
\end{eqnarray}

\vskip5pt
\noindent {\it Step 4. Proof of (\ref{e404}) }

From (\ref{e408}),
(\ref{e409}), (\ref{e413}) and (\ref{e416}), the inequality (\ref{e404}) follows at once.

\vskip 5pt

We now prove (\ref{e405}). Arbitrarily fix  $t\in \mathds{R}^+$.
 There is $n\in \mathds{N}$ so that
 \begin{eqnarray*}
 t\in (2nT, (2n+2)T]&=&(2nT, (2n+1)T]\bigcup((2n+1)T, (2n+{3}/{2})T]\\ \nonumber
 &\bigcup&
 ((2n+{3}/{2})T, (2n+2)T].
 \end{eqnarray*}

\noindent $(a)$ In the case when $t\in (2nT, (2n+1)T]$,
it follows from (\ref{wang4.6}) and (\ref{e404}) that
\begin{eqnarray}\label{e418}
\|y(t)\|&\leq& e^{\gamma_0 T}\|y(2nT)\|\leq e^{\gamma_0 T}e^{-2\gamma nT}\|y(0)\|\nonumber\\
&\leq& e^{\gamma_0 T}e^{-\gamma t}e^{2\gamma\big(\frac{t}{2}-nT\big)}\|y(0)\| \leq e^{(\gamma_0+\gamma) T}e^{-\gamma t}\|y(0)\|.
\end{eqnarray}

\noindent $(b)$ In the case where  $t\in ((2n+1)T, (2n+{3}/{2})T]$, it follows from (\ref{wang4.6}) and (\ref{e404}) that
\begin{eqnarray}\label{e419}
\|y(t)\|
&\leq& e^{\frac{\gamma_0T}{2}}\|y((2n+1)T)\|+\|\mathcal{F}_T\|
\int_{(2n+1)T}^te^{\gamma_0(t-s)}ds\|y(2n+{3}/{4})T\|\nonumber\\
&\leq& (1+({T}/{2})\|\mathcal{F}_T\|)e^{\big(\frac{3}{2}\gamma_0
+\frac{7}{4}\gamma\big)T}e^{-\gamma t}\|y(0)\|.
\end{eqnarray}

\noindent $(c)$ In the case that  $t\in ((2n+{3}/{2})T, (2n+2)T]$,
it follows from (\ref{wang4.6}) and (\ref{e404}) that
\begin{eqnarray}\label{e420}
\|y(t)\|
\leq e^{\frac{\gamma_0 T}{2}}\|y((2n+{3}/{2})T)\|
\leq (1+({T}/{2})\|\mathcal{F}_T\|)e^{(2\gamma_0+3\gamma)T}e^{-\gamma t}\|y(0)\|.
\end{eqnarray}

Now (\ref{e405}) follows from (\ref{e418}), (\ref{e419}) and (\ref{e420}).
This ends the proof of Theorem \ref{th401}.
\end{proof}

\subsection{Estimates on the feedback law}
The second main theorem of this paper gives some estimates on $\|\mathcal{F}_T\|$ in terms of $T$.
\begin{theorem}\label{theoremonestimate}
Let  $\mathcal{F}_T$ be given  by (\ref{e402}), with $T>0$ and $\gamma>0$. Then
there are two positive constants $C_1$ and $C_2$ depending on $\gamma_0$
$\Omega$, $\omega$, $\omega_1$ and $d$
so that
\begin{equation}\label{e406}
\alpha(T)\leq \|\mathcal{F}_T\|\leq e^{C_1(1+\frac{1}{T})+\frac{\gamma_0T}{4}}e^{C_2(1+\frac{1}{T})\gamma},
\end{equation}
where  $\alpha:(0,+\infty)\rightarrow (0,+\infty)$ satisfies that
\begin{eqnarray}\label{WANGGS4.11}
\alpha(T)=O({1}/{T})\ \mathrm{as}\ T\rightarrow 0; \ \ \ \
\alpha(T)=O(e^{-\frac{\lambda_1T}{4}})\ \mathrm{as}\ T\rightarrow +\infty.
\end{eqnarray}
\end{theorem}
\begin{proof}
 Arbitrarily fix  $T>0$ and $\gamma>0$.
  We divide the proof by the following three steps:
\vskip 5pt

\noindent {\it Step 1. We estimates  $h_j$ and $f_j$.}

First of all, by (\ref{wang4.1}), we can use  the Weyl's asymptotic theorem to find that
\begin{eqnarray}\label{e421}
N\leq \tilde{C}_0(1+\gamma_0+\frac{1}{T}+\gamma)^{d/2},
\end{eqnarray}
for some constant $\tilde{C}_0=C(\Omega, d)$.
Several facts are given in order.

Fact one: By (\ref{wang4.2}), (\ref{WANG4.3}),
  (\ref{e421}) and the  inequality:
 $\ln(1+x)< x$ for each $x>0$,
  we can verify that
\begin{eqnarray}\label{e422}
\sqrt{\lambda_M}
&<&\frac{2\tilde{C}_1+\sqrt{\frac{d}{2}}}{T}
+\frac{\sqrt{2\Big[\mathrm{ln}(9\sqrt{\tilde{C}_0})
+\frac{d}{4}(\gamma_0+\gamma)
+\tilde{C}_1(1+\gamma_0^\frac{2}{3})\Big]}}{\sqrt{T}}
\nonumber\\
&+&\sqrt{4\gamma+3\gamma_0}+\sqrt{\hat{c}_p}
\leq\frac{\tilde{C}_2}{T}+\tilde{C}_3.
\end{eqnarray}
Here, $\tilde{C}_1$ is given in (\ref{WANG4.3}),
$$
\tilde{C}_2=2\tilde{C}_1+\sqrt{\frac{d}{2}}
+\frac{5\Big[\mathrm{ln}(9\sqrt{\tilde{C}_0})
+\frac{d}{4}(\gamma_0+\gamma)
+\tilde{C}_1(1+\gamma_0^\frac{2}{3})\Big]}{2\sqrt{4\gamma+3\gamma_0}}, \tilde{C}_3=\frac{6}{5}\sqrt{4\gamma+3\gamma_0}+\sqrt{\hat{c}_p}.
$$

Fact two: Since $e^x-1>x$ for all $x>0$, we have that
\begin{equation}\label{e424}
\frac{\gamma_0}{1-e^{-\gamma_0 \frac{T}{2}}}=\gamma_0+\frac{\gamma_0}{e^{\gamma_0 \frac{T}{2}}-1}\leq \gamma_0+\frac{2}{T}.
\end{equation}

Fact three: By (\ref{ggsswang3.17}) (with $T_1=0$ and $T_2=T/2$) and (\ref{e422}), we have that
\begin{equation}\label{e425}
\frac{1}{\alpha_M}\leq \frac{1}{\int_0^{\frac{T}{2}}e^{-({\tilde{C}_2}/{T}
+\tilde{C}_3)^2t}dt}=\frac{({\tilde{C}_2}/{T}
+\tilde{C}_3)^2}{1-e^{-({\tilde{C}_2}/{T}+\tilde{C}_3)^2\frac{T}{2}}}\leq ({\tilde{C}_2}/{T}+\tilde{C}_3)^2+\frac{2}{T}.
\end{equation}

Fact four: By (\ref{wangpuchong4.5}),
we can apply Theorem \ref{th302} with (\ref{YUANGYUAN4.6})
to get that
\begin{equation}\label{wang4.37}
\|f_j\|_{L^2(\omega)}\leq
e^{\tilde{C}_1\big(1+\gamma_0^{\frac{2}{3}}
+\sqrt{\lambda_{M}}\big)}\Big(\frac{\gamma_0}{1-e^{-\gamma_0 \frac{T}{2}}}+\frac{\varepsilon_0}{\alpha_M}\Big)\|\xi_j\|.
\end{equation}

Now,  by (\ref{e422}), (\ref{e424}) and (\ref{e425}), we obtain that
\begin{eqnarray}\label{e426}
\|f_j\|_{L^2(\omega)}
&\leq& e^{\tilde{C}_1\big(1+\gamma_0^{\frac{2}{3}}
+\frac{\tilde{C}_2}{T}+\tilde{C}_3\big)}
\Big(\gamma_0+\tilde{C}_3^2+\frac{4+2\tilde{C}_2\tilde{C}_3}{T}
+\frac{\tilde{C}_2^2}{T^2}\Big)\|\xi_j\|\nonumber\\
&\leq& e^{\tilde{C}_4\big(1+\frac{1}{T}\big)}e^{\tilde{C}_5
\big(1+\frac{1}{T}\big)\gamma}\|\xi_j\|,
\end{eqnarray}
for some constants  $\tilde{C}_4$ and $\tilde{C}_5$  depending on $\gamma_0$, $\Omega$, $\omega$ and $d$, but independent of $\gamma$ and $T$.

Similarly, with the aid of Theorem \ref{th304}, we can verify that
\begin{eqnarray}\label{e427}
\|h_j\|_{L^2(\omega_1)}
\leq e^{\tilde{C}_6\big(1+\frac{1}{T}\big)+\frac{\gamma_0T}{4}}
e^{\tilde{C}_7\big(1+\frac{1}{T}\big)\gamma}\|\xi_j\|,
\end{eqnarray}
for some constants $\tilde{C}_6$ and $\tilde{C}_7$  depending on $\gamma_0$, $\Omega$, $\omega$ and $d$, but independent on $\gamma$ and $T$.

\vskip 5pt
{\it Step 2.  We get an upper bound for $\|\mathcal{F}_T\|$.}

By (\ref{e402}),  (\ref{e421}), (\ref{e426}) and (\ref{e427}), we can easily check that
\begin{eqnarray*}
&&\|\mathcal{F}_T(v)\|_{\omega}
=\Big\|\sum_{j=1}^{N}\langle \mathds{1}_{\omega_1}v, \
\mathds{1}_{\omega_1}h_{j}\rangle f_{j}\Big\|_{\omega}\nonumber\\
&\leq& N \max_{1\leq j\leq N}\|h_j\|_{\omega_1}\max_{1\leq j\leq N}\|f_j\|_{\omega}\|v\|_{\omega_1}\nonumber\\
&\leq& e^{\tilde{C}_8\big(1+\frac{1}{T}\big)+\frac{\gamma_0T}{4}}
e^{\tilde{C}_9\big(1+\frac{1}{T}\big)\gamma}\|v\|_{\omega_1}\;\;\mbox{for all}\;\; v\in L^2(\omega_1),
\end{eqnarray*}
where $\tilde{C}_8$ and $\tilde{C}_9$ are constants depending on $\gamma_0$, $d$, $\Omega$, $\omega$ and $\omega_1$, but independent of $\gamma$ and $T$. This leads to
 the second inequality in (\ref{e406}).
\vskip 5pt

\noindent {\it Step 3. We get a lower bound for  $\|\mathcal{F}_T\|$.}

We aim to show the first inequality in (\ref{e406}), as well as (\ref{WANGGS4.11}).
Arbitrarily fix $v\in L^2(\omega_1)$.
Let
\begin{equation}\label{wanG4.40}
\zeta^v:=\sum_{j=1}^{N}\langle v, h_j\rangle_{\omega_1}\xi_j.
\end{equation}
First we recall that   $h_1\neq 0$ (see (iii) of Remark \ref{wangremark4.1}).
Several facts are given in order.

Fact one:
Consider the problem $(SNP)$, given by (\ref{e301}) where
\begin{equation}\label{YUANYUAN4.44}
T_1=0,\;T_2=T/2,\;\zeta=\zeta^v,\;\varepsilon=\sqrt{N}\varepsilon_0,\;
M\;\mbox{given by}\;(\ref{wang4.2}),
\end{equation}
Write   $\hat{f^v}$ for its minimal norm control.
Then we have that
\begin{equation}\label{e432}
\|\mathcal{F}_T(v)\|_{\omega}\geq \|\hat{f^v}\|_{\omega}.
\end{equation}
Indeed, the solution $y(\cdot; \zeta^v, \mathcal{F}_T(v))$ to equation:
\begin{eqnarray*}
 \left\{\begin{array}{ll}
 y'(t)-Ay(t)=-\mathds{1}_\omega \mathcal{F}_T(v), t\in (0,\infty),\\
  y(0)=\zeta^v
\end{array}\right.
\end{eqnarray*}
satisfies that
\begin{eqnarray}\label{e431}
&&\|P_My({T}/{2}; \zeta^v, -\mathcal{F}_T(v))\|\nonumber\\
&=&\Big\|\sum_{j=1}^{N}\Big[\langle v,\; h_j\rangle_{\omega_1} P_M\Big(e^{A\frac{T}{2}}\xi_j+\int_0^{\frac{T}{2}}
e^{A(\frac{T}{2}-t)}\mathds{1}_\omega f_jdt\Big)\Big]\Big\|
\leq \sqrt{N}\varepsilon_0\|\zeta^v\|.
\end{eqnarray}
Hence, the control $-\mathcal{F}_T(v)$ is admissible to
the above problem $(SNP)$. This, along with  the optimality of $\hat{f^v}$,
       leads to (\ref{e432}).

Fact two: We have
that
\begin{eqnarray}\label{e433}
\|\hat{f^v}\|_{\omega}
  \geq \frac{1}{3\alpha_1}e^{-\gamma T}\big(1-\frac{1}{3}e^{-\big(\gamma+\frac{4}{3}\gamma_0+\hat{c}_p\big)T}\big)|\langle v, h_1\rangle_{\omega_1}|,\;\;\mbox{when}\;\;\langle v, h_1\rangle_{\omega_1}\neq 0,
\end{eqnarray}
where $\alpha_1$ is given by (\ref{ggsswang3.17}) (with $T_1=0$, $T_2=T/2$).
To show (\ref{e433}), we first claim that
\begin{equation}\label{wanG4.44}
\|P_Me^{A\frac{T}{2}}\zeta^v\|-\varepsilon_0\sqrt{N}\|\zeta^v\|>0,
\;\;\mbox{when}\;\;\langle v, h_1\rangle_{\omega_1}\neq 0.
\end{equation}
Indeed, since $M\geq N$ (see (\ref{wang4.4})), by (\ref{wanG4.40}),
 (\ref{wang4.1})
 and (\ref{wang4.3}),
  after some direct computations, we see that
\begin{eqnarray*}
&& \frac{1}{\alpha_1}\big(\|P_Me^{A\frac{T}{2}}\zeta^v\|-\varepsilon_0 \sqrt{N}\|\zeta^v\|\big)\nonumber\\
&\geq&  \frac{1}{\alpha_1}\Big[e^{-\lambda_N \frac{T}{2}}\Big(\sum_{j=1}^{N}|\langle v, h_j\rangle_{\omega_1}|^2\Big)^{\frac{1}{2}}-\varepsilon_0 \sqrt{N}\Big(\sum_{j=1}^{N}|\langle v, h_j\rangle_{\omega_1}|^2\Big)^{\frac{1}{2}}\Big]\nonumber\\
&\geq& \frac{1}{3\alpha_1}e^{-\gamma T}\big(1-\frac{1}{3}e^{-\big(\gamma+\frac{4}{3}\gamma_0+\hat{c}_p\big)T}\big)|\langle v, h_1\rangle_{\omega_1}|>0.
\end{eqnarray*}
This leads to (\ref{wanG4.44}).
Now by (\ref{wanG4.44}), we can apply  Theorem \ref{th302},
with (\ref{YUANYUAN4.44}),
to obtain  (\ref{e433}).

Fact three: Consider the problem $(INP)$ given by (\ref{e319})
with (\ref{YUANGYUAN4.7}) where $j=1$.
 Then $h_1$ is the minimal norm control to this problem.
Because of (\ref{wangpuchong4.5}), we can apply  Theorem \ref{th304} to the above $(INP)$ to see that
\begin{eqnarray}\label{e434}
 \|h_1\|_{\omega_1}
&\geq& e^{\frac{\lambda_1T}{4}}(\|e^{A\frac{T}{2}}\xi_1\|-\varepsilon_0)\nonumber\\
&=& e^{\frac{\lambda_1T}{4}}
\Big(e^{-\frac{\lambda_1T}{2}}
-\frac{1}{9\sqrt{N}}e^{-(2\gamma+\frac{3}{2}\gamma_0+\hat{c}_p)T}\Big)\geq \frac{8}{9}e^{\frac{-\gamma_0T}{4}}.
\end{eqnarray}
(Here we used that $-\gamma_0\leq\lambda_1<0$.)

Fact four: One can easily check that
\begin{equation}\label{WanG4.46}
\frac{1}{\alpha_1}\geq \frac{1}{\int_0^{T/2}e^{\gamma_0(\frac{T}{2}-t)}dt}=\frac{\gamma_0}{e^{\frac{\gamma_0T}{2}}-1}.
\end{equation}
\vskip 5pt

Now
it follows from
(\ref{e433}),
(\ref{e432}), (\ref{e434}) and (\ref{WanG4.46}) that
\begin{equation}\label{e436}
\|\mathcal{F}_T\|\geq \|\mathcal{F}_T(h_1/\|h_1\|_{\omega_1})\|_\omega\geq  \frac{16\gamma_0}{81(e^{\frac{\gamma_0T}{2}}-1)}e^{-(\frac{\gamma_0}{4}+\gamma) T}
:=m_1(T).
\end{equation}
One can easily check that
\begin{equation}\label{WANGPUCHONG4.49}
m_1(T)=O({1}/{T})\;\;\mbox{as}\;\;T\rightarrow 0.
\end{equation}
However,
$m_1(T)$ cannot play the role of $\alpha(T)$ in (\ref{e406}), since it does not
have the second property in (\ref{WANGGS4.11}).

To get the desired $\alpha(T)$, we need another lower bound for $\|\mathcal{F}_T\|$.
For this purpose, we arbitrarily fix $v\in L^2(\omega_1)$ with $\langle v, h_1\rangle_{\omega_1}\neq 0$.
By (\ref{e427}), it follows  that
\begin{eqnarray}\label{e438}
\Big(\sum_{j=1}^{N}|\langle v, h_j\rangle_{\omega_1}|^2\Big)^{\frac{1}{2}}
\leq \|v\|_{\omega_1}\sqrt{N}e^{\tilde{C}_6(1+\frac{1}{T})+\frac{\gamma_0T}{4}}e^{\tilde{C}_7(1+\frac{1}{T})\gamma}.
\end{eqnarray}
Meanwhile, by (\ref{wanG4.44}),  we can apply  Theorem \ref{th302},
with (\ref{YUANYUAN4.44}),
 to get that
\begin{eqnarray*}
\|\hat{f^v}\|_{\omega}\geq
\frac{1}{\alpha_1}(\|P_Me^{A\frac{T}{2}}\zeta^v\|-\varepsilon_0 \sqrt{N}\|\zeta^v\|)
\end{eqnarray*}
This, along with  (\ref{e438}), indicates that
\begin{eqnarray*}
\|\hat{f^v}\|_{\omega}
&\geq& \frac{1}{\alpha_1}\Big[\Big(\sum_{j=1}^{N}\big|e^{-\frac{\lambda_j T}{2}}\langle v, h_j\rangle_{\omega_1}\big|^2\Big)^{\frac{1}{2}}
-\sqrt{N}\varepsilon_0\Big(\sum_{j=1}^{N}\big|\langle v, \; h_j\rangle_{\omega_1}\big|^2\Big)^{\frac{1}{2}}\Big]\nonumber\\
&\geq& \frac{1}{\alpha_1}e^{-\lambda_1 \frac{T}{2}}|\langle v, h_1\rangle_{\omega_1}|
-\frac{\sqrt{N}}{9\alpha_1}e^{-\big(2\gamma+\frac{4}{3}\gamma_0+\hat{c}_p\big)T
+\tilde{C}_{10}(1+\frac{1}{T})(1+\gamma)+\frac{\gamma_0T}{4}}\|v\|_{\omega_1}.
\end{eqnarray*}
(Here $\tilde{C}_{10}$ is a constant depending on $\gamma_0, \Omega, \omega$, $\omega_1$  and $d$, but independent on $T$ and $\gamma$.)
   From this, (\ref{e432}), and  the first inequality in (\ref{e434}), we find that
\begin{eqnarray}\label{WANGGS4.52}
\|\mathcal{F}_T\|&\geq& \|\mathcal{F}_T(h_1/\|h_1\|_{\omega_1})\|_{\omega}\nonumber\\
&\geq& \frac{-\lambda_1e^{-\frac{3\lambda_1 T}{4}}}{e^{-\frac{\lambda_1 T}{2}}-1}+\frac{\lambda_1e^{-\frac{\lambda_1 T}{4}}}{9\sqrt{N}(e^{-\frac{\lambda_1 T}{2}}-1)}e^{-(2\gamma+\frac{4}{3}\gamma_0+\hat{c}_p)T}
\nonumber\\&-&\frac{\sqrt{N}}{9\alpha_1}
e^{-(2\gamma+\frac{13}{12}\gamma_0+\hat{c}_p)T
+\tilde{C}_{10}(1+\frac{1}{T})(1+\gamma)}\nonumber\\
&\geq& m_2(T),
\end{eqnarray}
where
\begin{equation}\label{WANGGS4.53}
\ \ m_2(T):=-\lambda_1e^{-\frac{\lambda_1 T}{4}}+\frac{\lambda_1}{1-e^{\frac{\lambda_1 T}{2}}}e^{-(2\gamma-\frac{\lambda_1 }{4})T}-\frac{\sqrt{N}}{9\alpha_1}
e^{-2\gamma T
+\tilde{C}_{10}(1+\frac{1}{T})(1+\gamma)}
\end{equation}
From (\ref{WANGGS4.53}), one can easily see that
\begin{equation}\label{e441}
m_2(T)=O(e^{-\frac{\lambda_1 T}{4}}), \mathrm{as}\ T\rightarrow +\infty.
\end{equation}
For each $T>0$, we let $\alpha(T):=\max\{m_1(T), m_2(T)\}$.
Then by (\ref{e436}) and (\ref{WANGGS4.52}), we find that this $\alpha(T)$ satisfies the first inequality in (\ref{e406}). By (\ref{WANGPUCHONG4.49}) and (\ref{e441}), we see that this $\alpha(T)$ satisfies  (\ref{WANGGS4.11}).
Thus we end the proof of Theorem \ref{theoremonestimate}.
\end{proof}

\end{document}